\documentclass{amsart}

\usepackage{graphicx}
\usepackage{enumerate}
\usepackage{amsmath,amssymb}
\usepackage{mathrsfs}

\newtheorem{maintheorem}{Theorem}

\newtheorem{theorem}{Theorem}[section]
\newtheorem{lemma}[theorem]{Lemma}
\newtheorem{proposition}[theorem]{Proposition}
\newtheorem{corollary}[theorem]{Corollary}

\theoremstyle{definition}
\newtheorem{remark}[theorem]{Remark}

\newtheorem{questioni}{Question}

\numberwithin{equation}{section}
\numberwithin{figure}{section}

\def\index{\mathrm{index}}
\def\loc{\mathrm{loc}}
\def\Diff{\mathrm{Diff}}
\begin{document}

%%%%%%%%%%%%%
\title[
$C^{2}$-robust heterodimensional tangencies 
]{
$C^{2}$-robust heterodimensional tangencies 
}

\author{Shin Kiriki} 
\address{Department of Mathematics, Kyoto University of Education, 
1 Fukakusa-Fujinomori, Fushimi, Kyoto, 612-8522, JAPAN}
\email{skiriki@kyokyo-u.ac.jp}

\author{Teruhiko Soma}
\address{Department of Mathematics and Information Sciences,
Tokyo Metropolitan University,
Minami-Ohsawa 1-1, Hachioji, Tokyo 192-0397, JAPAN}
\email{tsoma@tmu.ac.jp}


\date{\today}
%\date{version xx.xx}

%%%%%%%%%%%%%

\begin{abstract}
In this paper, 
we give sufficient conditions for the existence of $C^{2}$ robust heterodimensional tangency, 
and 
present a nonempty open set  in $\Diff^{2}(M)$ with $\dim M\geq 3$ 
each element of which 
has a non-degenerate 
heterodimensional tangency on  a $C^{2}$ robust heterodimensional cycle. 
\end{abstract}
\maketitle

\section{Introduction}\label{sec1}
The purpose of this paper is to show the existence of a large class 
of diffeomorphisms on $C^2$ manifold $M$ whose global dynamics is persistently non-dominated on cycles  by 
robust heterodimensional tangencies.  
To quickly explain our results, let us begin by recalling key words of non-hyperbolic diffeomorphisms: 
dominated splittings, homoclinic tangencies, 
and heterodimensional cycles.

The dominated splitting is perhaps a bottom structure obtained by weakening hyperbolicity,
which was first presented by Pliss \cite{Pl72}. 
Precisely,  
we say that a compact subset 
$\Lambda$ of a smooth closed manifold $M$ which is invariant by $f\in \Diff^{r}(M)$ has a \emph{dominated splitting} 
if  $T_{\Lambda}M$ is expressed by the direct sum of $Df$-invariant subbundles $E$ and $F$ 
as $T_{\Lambda}M=E\oplus F$ 
whose fibers have constant dimensions
and 
there are $C>0$ and $\lambda>1$ such that 
for any integer $n > 0$, $x\in\Lambda$ 
and any pair of unitary vectors $(u, v)\in E(x)\times F(x)$, 
$$
\left\| Df^{n}(x)u\right\|  \left\| Df^{n}(x)v \right\|^{-1}<C\lambda^{-n}.
$$
If $Df^{n}$ is uniformly contracting on $E$ and  expanding on $F$, 
$\Lambda$ is called \emph{hyperbolic}.

Let $P$ be a saddle point for  $f\in \Diff^{r}(M)$. 
A point $Y$ of the intersection between the stable  manifold  $W^{s}(P)$
and the unstable manifold $W^{u}(P)$ is a \emph{homoclinic tangency} of $P$ 
if  
$$P\neq Y,\quad T_{Y}M\neq T_{Y} W^{s}(P)+ T_{Y}  W^{u}(P).$$
Obviously,  if $f$ has a homoclinic tangency, then
any $f$-invariant set containing the orbit of the tangency does not have any dominated splitting. 
Moreover, it is worth noting that Wen  \cite[Theorem A]{W02} 
proved that the inverse is also true in the  $C^{1}$ topology if $f$ is restricted to the preperiodic or prehomoclinic sets, that is, if a dominated splitting cannot be defined on such $f$-invariant sets, then there is a diffeomorphism arbitrarily  $C^{1}$  close to $f$ which has a homoclinic tangency, see also \cite{Gou10}.
On the other hand, the presence of dominated splittings is equivalent to 
the absence of infinitely many periodic sinks or sources 
 in $C^{1}$ generic  diffeomorphisms, see \cite{ABC06}.

From the definition, a tangency associated with a periodic point  is easily broken by a generic perturbation.
However,  Newhouse showed in \cite{N79} that   
there is an open set in $\Diff^{2}(M)$ with $\dim(M)=2$ arbitrarily $C^{2}$ close to $f$ in which any 
diffeomorphism has a robust homoclinic tangency associated with  some nontrivial set containing 
the continuation of the periodic point.  On the other hand, 
Moreira showed in \cite{Mo11} that the similar result does not hold for generic subset in $\Diff^{1}(M)$ with $\dim(M)=2$, 
see comments at the end of this section.
The Newhouse result was extended to the higher dimensional cases with the $C^{2}$ topology
in \cite{GTS93, PV94,R94,T01}.  
Among others, several results of \cite{PV94}
play important roles in this paper.  Thus, we need to work at least in the $C^{2}$ topology.

Another known cause of non-hyperbolicity is the presence of heterodimensional cycles. 
For a diffeomorphism $f$ on a smooth manifold $M$ of $\dim(M)\geq 3$, 
we say that  $f$ has a \emph{heterodimensional cycle}  associated with two saddle points $P$ and $Q$ if 
 $$\index(P)\neq  \index(Q),\quad 
 W^{s}(P)\cap  W^{u}(Q) \neq \emptyset,\ \quad 
W^{s}(Q)\cap  W^{u}(P) \neq \emptyset,
 $$
 where $\index(P)$ stands for the dimension of the unstable manifold of $P$. 
By the first index condition, one of these two intersections,
say $W^{s}(Q)\cap  W^{u}(P)$, satisfies
$$
 \dim(T_{X}W^{s}(Q))+ \dim(T_{X}W^{u}(P))<\dim(T_{X}M)
 $$
 for every $X\in W^{s}(Q)\cap  W^{u}(P)$. 
 On the other hand, 
the  intersection $W^{s}(P)\cap W^{u}(Q)$ satisfies 
$$
\dim(T_{Y}W^{s}(P))+ \dim(T_{Y}W^{u}(Q)) > \dim(T_{Y}M)
$$ 
for every $Y\in W^{s}(P)\cap W^{u}(Q)$. 
Thus the condition may permit transverse points.
The set of such points is denoted by $W^{s}(P)\pitchfork W^{u}(Q)$.
A non-transverse point  
$$Y \in (W^{s}(P)\cap W^{u}(Q) ) \setminus (W^{s}(P)\pitchfork W^{u}(Q))$$ 
is called a \emph{heterodimensional tangency}.
We say that such a heterodimensional tangency is \emph{strict} if
$$
T_{Y} W^{s}(P)=T_{Y} W^{u}(Q).
$$
Observe that, when $\dim(M)=3$, any heterodimensional tangency is of strict type.

The condition for diffeomorphisms being away from homoclinic tangency leads to a dominated splitting 
even for heterodimensional cycle associated with $P$ and $Q$. 
In fact, D\'iaz and Rocha showed in \cite[Theorem A]{DR01} that 
if  the diffeomorphism  is $C^{r}$ away 
from diffeomorphisms  having  
homoclinic tangencies associated  with   continuations of $P$ and $Q$, 
it can be $C^{r}$-approximated by one with an invariant set containing the continuations  which 
has a (strong partially hyperbolic) dominated splitting.  
On the other hand, it is known that, under some volume condition, the absence of dominated splittings brings on  
heterodimensional cycles in $C^{1}$ generically \cite{Shino11}. 

In the present paper, we  focus attention on heterodimensional  tangencies which 
interfere with dominated splittings. 
As for heterodimensional tangencies of $3$-dimensional diffeomorphisms, 
several results have been already presented as follows.
\begin{description}
\item[$C^{1}$ Newhouse phenomenon \cite{DNP06}]  
If $f$ has  a parabolic heterodimensional tangency between 
$P$ and $Q$ which are persistently linked, 
it can be $C^{1}$-approximated by a diffeomorphism 
having either infinitely many sinks or sources. Moreover, 
if the Jacobian of $f$ at $P$ is smaller than one and 
that of $Q$ is greater than one, 
it can be $C^{1}$-approximated by a diffeomorphism 
having  infinitely many nontrivial minimal Cantor sets.
\item[Strange attractors, $C^{2}$ Newhouse phenomenon \cite{KNS10}] 
For any  generic  $2$-parameter family $\{f_{\mu,\nu}\}$ with $f_{0,0}=f$ 
which has a nondegenerate  heterodimensional tangency associated  with $P$ and $Q$, 
there exists an open set of subfamilies  
exhibiting infinitely many homoclinic tangencies of the continuation of $P$ 
which unfold generically. It follows that one can detect 
a positive Lebesgue measure set $\mathcal{A}$ in the 
$\mu\nu$-plane arbitrarily near $f$  such that for any $(\mu,\nu) \in \mathcal{A}$,  
$f_{\mu,\nu}$ exhibits non-hyperbolic strange attractors. Moreover, one can detect 
an open set of diffeomorphisms which have $C^{r}$-robust homoclinic tangency for  $r\geq 2$. 
\item[Renormalization, robust cycle \cite{DKS-pre}]
Under appropriate conditions, one can obtain a renormalization in a neighborhood of the heterodimensional tangency 
associated  with $P$ and $Q$ such that its   
return maps converge to the H\'enon-like family with the center unstable direction  
and admit blender-horseshoes. 
Using this fact, for every $r\geq1$, one can detect $C^{r}$-robust cycles associated with 
the blender-horseshoe and the continuation of $P$ arbitrarily $C^{r}$-close to $f$.
\end{description}
\medskip

By the way,  though the above works start from the heterodimensional tangency on a heterodimensional cycle,
its robustness is not discussed. However, 
by a little perturbation, heterodimensional tangencies are  broken as well as homoclinic tangencies.
Thus, considering Newhouse's works, we ask naturally:

\begin{questioni}\label{q_1}
Does there exist an open set of diffeomorphisms 
which have heterodimensional tangencies on 
heterodimensional cycles?
\end{questioni}

This paper will be devoted to answering the question and studying related topics.
To present our results, we have to introduce 
several   terminologies.

A \emph{basic set} for $f$ is a compact hyperbolic transitive locally maximal $f$-invariant subset in $M$.
A basic set is \emph{nontrivial} if it is not an periodic orbit.
The dimension of the unstable bundle on a basic set $\Lambda$ for $f$ is called the \emph{unstable index} and denoted by $\index(\Lambda)$.
We say that a diffeomorphism $f\in\Diff^{r}(M)$ 
has a \emph{heterodimensional  cycle} associated with basic sets $\Lambda_{0}$ and $\Lambda_{1}$ if 
$$\index(\Lambda_{0})\neq  \index(\Lambda_{1}),\quad 
W^{u}(\Lambda_{0})  \cap W^{s}(\Lambda_{1})  \neq \emptyset,\ \quad 
W^{s}(\Lambda_{0})\cap  W^{u}(\Lambda_{1}) \neq \emptyset.
$$
A \emph{heterodimensional tangency} between $W^u(q)$ 
and $W^s(p)$ with  $q\in \Lambda_{0}$ and $p\in \Lambda_{1}$  is 
defined as in the case of periodic points $Q, P$.
Now let us suppose $\index(\Lambda_{0})>  \index(\Lambda_{1})$ and 
define that 
 $f$ has a $C^{r}$ \emph{robust heterodimensional tangency}  associated with $\Lambda_{0}$ and $\Lambda_{1}$ 
 if  there exists a $C^{r}$ neighborhood $\mathcal{U}$ of $f$ such that, for any $g\in \mathcal{U}$, 
there exist the continuations $\Lambda_{0,g}$ of $\Lambda_{0}$ and $\Lambda_{1,g}$ of $\Lambda_{1}$ which contain points  
$\tilde q\in \Lambda_{0, g}$ and $\tilde{p}\in \Lambda_{1,g}$
such that 
\begin{itemize}
\item  $W^{u}(\tilde q)$ and $W^{s}(\tilde p)$ contain a  heterodimensional tangency.
\end{itemize}
Moreover, the heterodimensional tangency is in a $C^{r}$ \emph{robust cycle} 
if there exist  
$\hat q \in \Lambda_{0,g}$ and $\hat p \in \Lambda_{1,g}$ such that 
\begin{itemize}
\item $W^{u}(\hat p)\cap W^{s}(\hat q)$ is nonempty.
\end{itemize}

In what follows of this paper, we only discuss the case 
of 
$$\dim M=d\geq 3,\ 
\index(\Lambda_{0})=d-1,\ 
\index(\Lambda_{1})=1.
$$
Under this assumption,  
a heterodimensional tangency $Y\in W^{u}(\tilde q)\cap W^{s}(\tilde p)$ can be characterized 
from a topological viewpoint as follows. 
Consider a local coordinate $(x_{1},\ldots, x_{d})$ on a neighborhood of $Y$ 
with $Y=\mathbf{0}$ such that 
a small $(d-1)$-disk in $W^{s}(\tilde p)$ containing $Y$ is presented by the graph of 
the constant function $x_{d}=0$, and 
a small $(d-1)$-disk in $W^{u}(\tilde q)$ containing $Y$ is given as the graph of a 
$C^{2}$ function $u:\mathbb{R}^{d-1}\to \mathbb{R}$ with $x_{d}=u(x_{1},\ldots, x_{d-1})$ 
and $\dfrac{\partial u}{\partial x_i}(\mathbf{0})=0$ for $i=1,\dots,d-1$. 
We say that the tangency is \emph{non-degenerate} if the Hessian matrix $Hu(\mathbf{0})=\left(\dfrac{\partial^{2}u(\mathbf{0})}{\partial x_{i}\partial x_{j}}\right)$ 
is regular.
A non-degenerate tangency is \emph{elliptic} if all the eigenvalues of the Hessian matrix have the 
same sign, and otherwise \emph{hyperbolic}, see \cite{KNS10} for the case of $d=3$. 
\medskip

From now on, we suppose that $M$ is a closed $C^{2}$ manifold  
with Riemannian metric and $\Diff^{2}(M)$ is 
the space of all $C^{2}$ diffeomorphisms on $M$ 
endowed with the $C^{2}$ topology. 
For a given saddle periodic point $P$ and 
a basic set $\Lambda$ of $f\in \Diff^2(M)$, 
$\mathrm{per}(P)$ denotes  the minimum period of $P$, 
and 
$\Lambda_{g}$ denotes the continuation of $\Lambda$ for any $g$ sufficiently close to $f$ in $\Diff^2(M)$.

\begin{maintheorem}\label{theoremA}
Let $f$ be an element in $\Diff^{2}(M)$ with $\dim M=d\geq 3$ which has nontrivial basic sets $\Lambda_0$, $\Lambda_1$ 
and saddle periodic points $Q\in \Lambda_0$,  $P\in \Lambda_1$ 
such that 
\begin{enumerate}[\rm (1)]
\item 
$Df^{\mathrm{per}(P)}(P)$ has eigenvalues satisfying 
$$
|\alpha_{1}|\leq  \cdots  \leq |\alpha_{d-2}| < |\alpha_{c}|<1<|\alpha_{u}|, \  |\alpha_{c}\alpha_{u}|<1,
$$
 $Df^{\mathrm{per}(Q)}(Q)$ has eigenvalues satisfying 
$$
|\beta_{s}|<1< |\beta_{c}| <|\beta_{1}| \leq  \cdots  \leq |\beta_{d-2}|.
$$
\item $W^{u}(Q)\cap W^{s}(P)$  contains a heterodimensional tangency of elliptic type 
and $W^{u}(P)\cap W^{s}(Q)$ contains a nontransverse intersection.
\end{enumerate}
Then, there is a nonempty open set $\mathcal{O}\subset \Diff^{2}(M)$  
whose closure contains $f$ and that satisfies following condition: 
for every  $g\in \mathcal{O}$, there exists a  nontrivial basic set $\Lambda_{2,g}$ of index one 
such that 
$g$ has a heterodimensional tangency of elliptic type    
associated with $\Lambda_{2,g}$ and the continuation $\Lambda_{0,g}$ of $\Lambda_0$. 
\end{maintheorem}

The saddle periodic point $P$ in (1) is called \emph{sectionally dissipative} in \cite{PV94}, and  
$\alpha_{c}$ and $\beta_{c}$ are called  the \emph{real central contracting} and 
\emph{real central expanding} eigenvalues, respectively, on the cycle.  
Though the sectionally dissipative condition in (1) might not be indispensable 
if one would use \cite{R94} instead of \cite{PV94}, it is sufficient to prove Theorems \ref{theoremB} and 
\ref{theoremC}.
So we will work under the condition to avoid miscellaneous difficulties.
Note that the condition (1) implies that 
the heterodimensional tangency in the condition (2) is of strict type.

Theorem \ref{theoremA} says nothing about  
whether or not the $C^{2}$ robust heterodimensional tangency is in a heterodimensional cycle
associated with $\Lambda_{0,g}$ and $\Lambda_{2,g}$. 
However, one can obtain an affirmative answer to Question \ref{q_1} for $r\geq 2$ by supposing that the diffeomorphism in Theorem \ref{theoremA} has a certain basic set 
called a \emph{$cu$-blender} \cite[Definition 3.1]{BD12, BDK12}  which will be given in Section \ref{S7}.  

\begin{maintheorem}\label{theoremB}
There exists a
nonempty open set in $\Diff^{2}(M)$  each element of which 
has a heterodimensional tangency  of elliptic type on a heterodimensional cycle. 
\end{maintheorem}

Now we consider the relation between 
heterodimensional tangencies and the absence of dominating splittings.  
Trivially, if  a heterodimensional cycle contains a heterodimensional tangency, then  
there is no dominating splitting on the cycle.    
As mentioned above, in the $C^{1}$ topology, 
the existence of homoclinic tangencies and the absence of dominating splittings are 
synonymous in the sense of Wen \cite{W02}.  
In the $C^{2}$  topology, the result corresponding to that in \cite{W02} is not known yet.
Thus, the following question makes sense.

\begin{questioni}\label{q_2}
Can a diffeomorphism with a heterodimensional cycle 
which does not admit any  dominating splitting  be $C^{2}$-approximated by a diffeomorphism 
having  a heterodimensional tangency on a cycle?
\end{questioni}

In the following theorem, we will present open conditions on $C^{2}$ diffeomorphisms  
under which 
Question \ref{q_2} is solved affirmatively.

Let $\Lambda_{0}$ and $\Lambda_{1}$ be nontrivial basic sets
for $f\in \Diff^{2}(M)$ with $\index(\Lambda_{0})=d-1$ and $\index(\Lambda_{1})=1$, where 
$d=\dim(M)\geq 3$ . 
We suppose furthermore that 
$\Lambda_{0}$  is a blender-horseshoe
and 
$\Lambda_{1}$ is a \emph{sectionally dissipative} basic set 
\emph{with a real central contracting direction},
that is, 
every periodic point in $\Lambda_{1}$
satisfies the same condition as in (1) of Theorem \ref{theoremA}.

\begin{maintheorem}\label{theoremC}
Let $f$ be a diffeomorphism satisfying the above conditions.
If $f$ has a spherical heterodimensional intersection on the heterodimensional cycle
associated with $\Lambda_{0}$ and $\Lambda_{1}$ 
and satisfies 
the $C^2$ open conditions given in Section \ref{S8}, then  $f$ is $C^{2}$-approximated by a diffeomorphism having 
a $C^{2}$-robust cycle with a robust heterodimensional tangency.
\end{maintheorem}

\noindent
Here we say that $f$ has a \emph{spherical heterodimensional intersection} 
associated with $\Lambda_{0}$ and $\Lambda_{1}$ if there are $q_{1}\in \Lambda_{0}$ 
and $p_{1}\in \Lambda_{1}$ such that $W^{u}(q_{1})\pitchfork W^{s}(p_{1})$ contains  
a $(d-2)$-dimensional sphere.

Note that 
the cycle considered in Theorem \ref{theoremC} is 
critical in the sense of \cite{DR92}. 
Moreover,  in the case of $d=3$, D\'iaz, Nogueira and Pujals proved that 
the homoclinic classes of $\Lambda_{0}$ and $\Lambda_{1}$ in Theorem \ref{theoremC} do not admit any dominated splitting \cite[\S 2.3]{DNP06}.

%Under the sectionally dissipative condition of $\Lambda_{1}$ 
%with a real central contracting direction, one has  a $(d-2)$-dimensional 
%strong stable foliation $\mathcal{F}^{ss}(p_{1})$ on $W^s(p_1)$. 
%Hence, 
%since the  spherical heterodimensional intersection 
%gives a $(d-2)$-dimensional sphere in $W^{u}(q_{1})\pitchfork W^{s}(p_{1})$, 
%it has at least two tangencies with 
%leaves of $\mathcal{F}^{ss}(p_{1})$, see Figure \ref{fig6_4} 
%for further details given in proof of Theorem \ref{theoremC}. 
%Thus, the above heterodimensional cycle containing a spherical intersection is 
%critical in the sense of \cite{DR92}.

We comment here briefly on problems related to our theorems.  
In the present paper, we have obtained an affirmative answer to Question \ref{q_1} 
about the existence of robust heterodimensional tangencies
in the $C^{r}$ topology for any $r\geq 2$.
Thus, it is natural to ask whether there exist robust heterodimensional tangencies in the $C^{1}$ topology.
An example is already known, 
which exhibits a $C^{1}$-robust \emph{homoclinic} tangency  in the dimension at least $3$, see \cite{As08}.   
However,  Moreira \cite{Mo11} showed that 
any two regular Cantor sets $K_1$ and $K_2$ have continuations $\tilde K_1$, $\tilde K_2$ 
which are arbitrarily $C^{1}$ close to the originals and disjoint to each other.
This implies that $C^{1}$-robust homoclinic tangencies
can not be possible for $2$-dimensional generic diffeomorphisms. 
Similarly, we think that 
\emph{$C^{1}$-robust heterodimensional tangencies 
could not be possible for generic diffeomorphisms in any dimension greater than two}.

\subsection*{Outline of Proof of Theorem \ref{theoremA}}
At the end of this section,  we will outline the proof of the main theorem.
The diffeomorphism $f$ in Theorem \ref{theoremA} has basic sets $\Lambda_0$, $\Lambda_1$ with 
$\mathrm{index}(\Lambda_0)=d-1$, $\mathrm{index}(\Lambda_1)=1$ and 
periodic points $Q\in \Lambda_0$ and $P\in \Lambda_1$ such that $W^u(Q,f)$ and $W^s(P,f)$ have a 
heterodimensional tangency $Y$ of elliptic type and $W^u(P,f)$ and $W^s(Q,f)$ have a nontransverse 
intersection $X$, see Figure \ref{fig1}.
First we perturb $f$ near $X$ slightly so that $W^u(P,f)$ returns near $Y$ via a neighborhood $U_Q$ of $Q$.
Perturb $f$ again in a small neighborhood $U_Y$ of $Y$ and get a diffeomorphism $g_1$ such that $W^u(P,g_1)$ and $W^s(P,g_1)$ have a 
homoclinic tangency in $U_Y$.
Moreover, the heterodimensional tangency $Y\in W^u(Q,f)\cap W^s(P,f)$ is broken and becomes a $(d-2)$-sphere $S_0^{d-2}$ in $W^u(Q,g_1)\pitchfork W^s(P,g_1)$ (Proposition \ref{prop2.1}).
By applying Palis-Viana's result in \cite{PV94} which is based on Palis-Takens \cite{PT93}, we have 
a diffeomorphism $g_2$ arbitrarily close to $g_1$ which has a new basic set $\Lambda_2$ in $U_Y$ homoclinically related to $\Lambda_1$ and such that its stable thickness $\tau^s(\Lambda_2)$ 
is very large (Proposition \ref{prop4.1}).
Let $A_0^u$ be a small tubular neighborhood of $S_0^{d-2}$ in $W^u(Q,g_2)$.
The iterated forward images of $A_0^u$ by $g_2$ converge to $W^u(P,g_2)$ and hence return to $U_Y$, see 
Figures \ref{fig41} and \ref{fig4.2}.
Then we perturb $g_2$ in $U_Y$ and get a diffeomorphism $g_3$ such that $W^u(q,g_3)$ and 
$W^s(p,g_3)$ have a heterodimensional tangency $r$ for some $q\in \Lambda_0$ and $p\in \Lambda_2$ 
(Proposition \ref{prop4.3}), see Figure \ref{fig4.2} again.
By Gap Lemma \cite{N79,PT93} together with the largeness of $\tau^s(\Lambda_2)$, one can show that,  
for any diffeomorphism $g$ sufficiently near $g_3$, $W^u(\Lambda_0,g)$ and $W^s(\Lambda_2,g)$ have a 
heterodimensional tangency in $U_Y$.
Theorem \ref{theoremA} follows directly from this fact.

Note that we start with a heterodimensional tangency associated with $Q\in \Lambda_0$ 
and $P\in \Lambda_1$.
However, our robust heterodimensional tangencies are associated with $\Lambda_{0,g}$, $\Lambda_{2,g}$ but not with $\Lambda_{0,g}$, $\Lambda_{1,g}$.

\section{Coexistence of spherical intersections and tangencies}\label{sec2}
Suppose that $M$ is a $C^2$ manifold of dimension $d\geq 3$.
The purpose of this section is to find an element of $\Diff^2(M)$ arbitrarily $C^{2}$ close to 
the diffeomorphism given in  Theorem \ref{theoremA} which has simultaneously a spherical heterodimensional  intersection and a homoclinic tangency.

Let $f$ be an element of $\Diff^2(M)$ with periodic points $P$ and $Q$ satisfying (1) and (2) in Theorem \ref{theoremA}.
For simplicity, we may suppose that 
$\mathrm{per}(P)=\mathrm{per}(Q)=1$
if necessary replacing $f$ by $f^n$ for some $n\in \mathbb{N}$
 and 
$f$ is $C^{2}$ linearizable in small neighborhoods $U_{P}$ of $P$ and $U_{Q}$ of $Q$. 
Then, by the condition (1), $f$ can be written  in $U_{P}$ as 
\begin{equation}\label{linearmap_P}
f(\boldsymbol{x},y, z)=(A_{s}\boldsymbol{x}, \alpha_{c}y,  \alpha_{u}z)
\end{equation}
where 
$\boldsymbol{x}=(x_{1},\cdots,x_{d-2})\in \mathbb{R}^{d-2}$,  
$y, z\in \mathbb{R}$ and 
$A_{s}$ is a regular $(d-2)$-matrix with eigenvalues $\alpha_{1}, \dots, \alpha_{d-2}$ 
satisfying
\begin{equation}\label{P-eigenvalues}
|\alpha_{1}|\leq  \cdots  \leq |\alpha_{d-2}| < |\alpha_{c}|<1<|\alpha_{u}|,\quad |\alpha_{c}\alpha_{u}|<1.
\end{equation}
On the other hand in $U_{Q}$, 
$f$ can be written as 
\begin{equation}\label{linearmap_Q}
f(x,y, \boldsymbol{z}) =(\beta_{s}x,\beta_{c}y, B_{u}\boldsymbol{z}), 
\end{equation}
where  
$x, y\in \mathbb{R}$,  
$\boldsymbol{z}=(z_{1},\cdots,z_{d-2})\in \mathbb{R}^{d-2}$ and 
$B_{u}$ is a regular $(d-2)$-matrix with eigenvalues $\beta_{1},\dots, \beta_{d-2}$ 
satisfying 
\begin{equation}\label{Q-eigenvalues}
|\beta_{s}|<1< |\beta_{c}| <|\beta_{1}| \leq  \cdots  \leq |\beta_{d-2}|.
\end{equation}

We say that a nontransverse intersection $X\in W^{u}(P)\cap W^{s}(Q)$  is \emph{quasi-transverse} if 
it satisfies 
$$
T_{X}W^{s}(Q)+T_{X}W^{u}(P)=T_{X}W^{s}(Q)\oplus T_{X}W^{u}(P),
$$
see \cite{DR01}. 
For the nontransverse intersection $X \in W^{u}(P)\cap W^{s}(Q)$ and the heterodimensional tangency $Y\in W^{s}(P)\cap W^{u}(Q)$ of elliptic type given in (2) of Theorem \ref{theoremA},
we may furthermore assume that  $f$ satisfies the following situations without loss of generality.
See Figure \ref{fig1}.

\begin{remark}[on linearizing coordinates]\label{rmk1}
\begin{enumerate}[(i)]
\item 
$X$ is located at 
$(1, 0,\mathbf{0})$ with respect to the linearizing coordinate on $U_{Q}$, 
where $\mathbf{0}=(0,\ldots,0)\in \mathbb{R}^{n-2}$.
Moreover, $T_{X} W^{u}(P)$ and the eigenspace associated with $\beta_{c}$ are linearly independent. 
\item 
$Y$ is located at  $(\boldsymbol{1}, 1,  0)$ 
with respect to the linearizing coordinate on $U_{P}$ where $\boldsymbol{1}=(1,\ldots,1)\in \mathbb{R}^{d-2}$.
\item 
There exist integers $N_{1}, N_{2}>0$ such that 
$\tilde{X}=f^{-N_{1}}(X)\in W^{s}_{\loc}(P)=( \mathbf{0}, 0, 1)$ and 
$\tilde{Y}=f^{-N_{2}}(Y)\in W^{u}_{\loc}(Q)=(0, 1, \boldsymbol{1})$.
\end{enumerate}
\end{remark}

%%%%%%%%%%%%%%%%%%%%%%%%%%%%%%%%%%%%%%%
\begin{figure}[hbt]
\centering
\scalebox{0.53}{\includegraphics[clip]{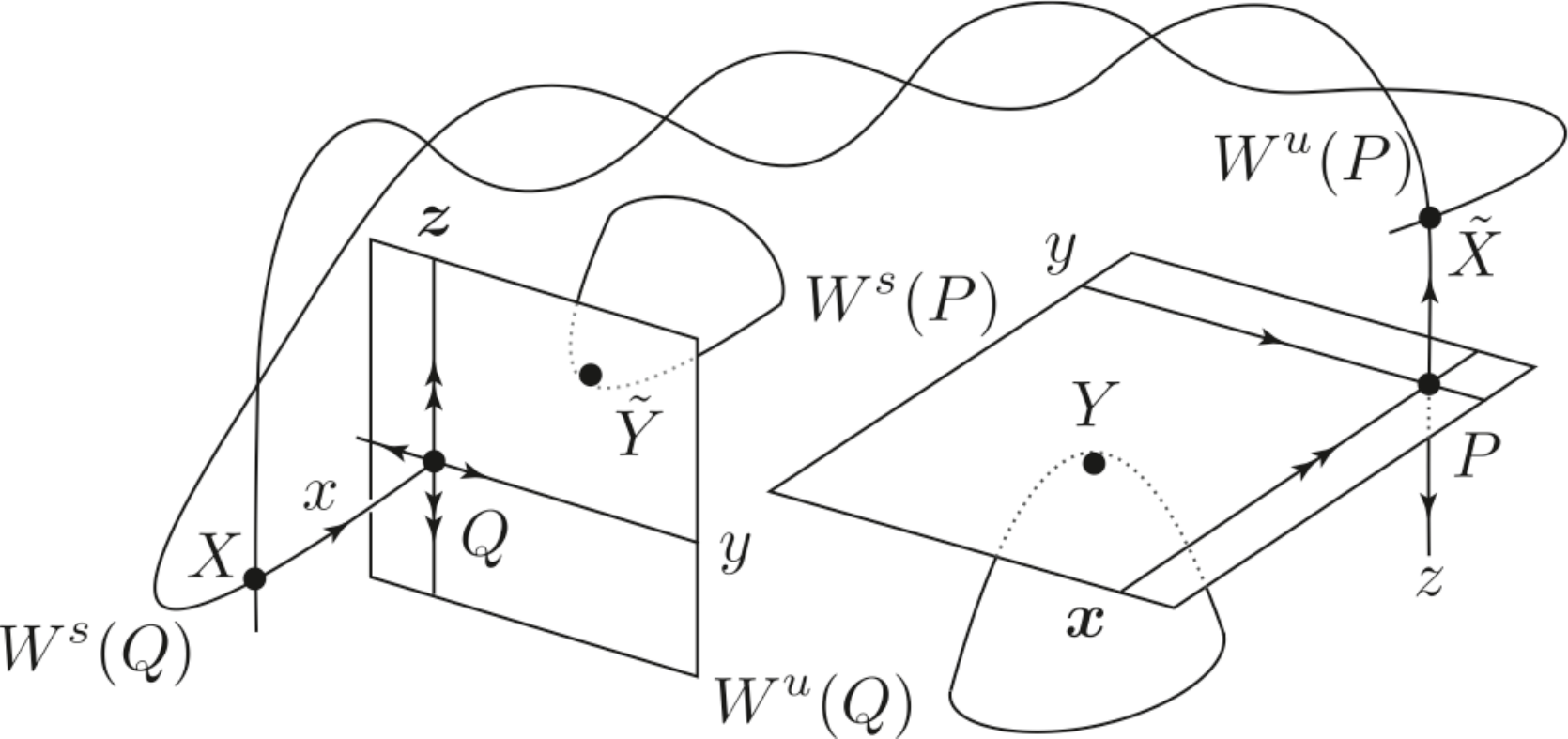}}
\caption{Heterodimensional tangency of elliptic type on a cycle.
}
\label{fig1}
\end{figure}
%%%%%%%%%%%%%%%%%%%%%%%%%%%%%%%%%%%%%%%

We now prove the following proposition under the linearizing coordinates 
of Remark \ref{rmk1}. 
Note that, even if $Y\in W^{ss}(P)$ and $\tilde Y\in W^{uu}(Q)$,  
one can show similarly the following proposition 
by using the linearizing  coordinate  
with $Y=(\boldsymbol{1}, 0,  0)$ and 
$\tilde{Y}=(0, 0, \boldsymbol{1})$ respectively.

\begin{proposition}\label{prop2.1}
There exists a $g_1\in \Diff^{2}(M)$ arbitrarily $C^{2}$ close to the above $f\in \Diff^{2}(M)$ such that 
$g_1$ has simultaneously the transverse intersection 
$W^{s}(P, {g_1})\pitchfork W^{u}(Q, {g_1})$ containing  a $(d-2)$-sphere in $W^{u}(Q, {g_1})$
and a quadratic homoclinic tangency associated with $P$.
\end{proposition}
Here a $(d-2)$-\emph{sphere} in $W^{u}(Q, {g_1})$ 
means the  
boundary of a $C^2$ embedded $(d-1)$-disk in $W^u(Q,g_1)$.

For the proof of the proposition, let us 
prepare a suitable parameterized family in $\Diff^{2}(M)$ containing $f$. 
For a sufficiently small $\delta>0$,
let $U_{X}$ and $U_{Y}$ 
be  the $2\delta$-neighborhoods 
of $X=(1,0,\mathbf{0})$ and $Y=(\mathbf{1},1,0)$ respectively.
To define local perturbations of $g_1$ in  
$U_{X}$  and $U_{Y}$, 
consider the functions on $U_X$, $U_Y$ defined as  
\begin{equation*}
H_{X}(x,y,\boldsymbol{z}) = h(x-1)  h(y) \prod_{i=3}^{n}h(z_{i}) ,\
H_{Y}(\boldsymbol{x},y,z) = \prod_{i=1}^{n-2}h(x_{i}-1) h(y-1) h(z),
\end{equation*}
where $h$ is a $C^{2}$ bump function  on $\mathbb{R}$ 
satisfying 
$$
\begin{cases}
h(t)= 0\ & \mbox{if}\  2\delta \leq \vert t\vert; \\
0< h(t) <1\ & \mbox{if}\  \delta<|t|<2\delta;\\
h(t)= 1\  &  \mbox{if}\  \vert t\vert  \leq \delta.
\end{cases} 
$$
Fix $\delta_0>0$ which is sufficiently smaller than $\delta$, e.g.\ $\delta_0=\delta/100$.
Let $\{\varphi_{\mu,\nu}\}$ $(-\delta_0<\mu,\nu<\delta_0)$ be the family  of perturbations in $\Diff^{2}(M)$ given by 
\begin{equation}\label{perturb1}
\begin{cases}
\varphi_{\mu,\nu}(x,y,\boldsymbol{z})=(x, y+\nu H_{X}(x,y,\boldsymbol{z}), \boldsymbol{z} ) & \text{if $(x,y,\boldsymbol{z})\in U_{X}$}\\
\varphi_{\mu,\nu}(\boldsymbol{x},y,z)= (\boldsymbol{x}, y, z+\mu H_{Y}(\boldsymbol{x},y,z)) & \text{if  $(\boldsymbol{x},y,z)\in U_{Y}$}\\
\varphi_{\mu,\nu}=\mathrm{id}_{M\setminus U_X\cup U_Y} & \text{otherwise.}
\end{cases} 
\end{equation}
Using the $\{\varphi_{\mu,\nu}\}$, we define the $2$-parameter family $\{f_{\mu,\nu}\}$ by 
\begin{equation}\label{eqn5}
f_{\mu,\nu}=\varphi_{\mu,\nu}\circ f.
\end{equation}
For the definition, it is clear that $f_{\mu,\nu}\to f$ in the $C^{2}$ topology as $\mu,\nu\to 0$. 

\begin{remark}[about notations]
\begin{itemize}
\item Since $P\not\in U_{X}$ and $Q\not\in U_{Y}$, 
the continuations of these points 
satisfy $P_{\mu, \nu}=P$ and $Q_{\mu, \nu}=Q$ for every $\mu,\nu$ with $-\delta_0<\mu, \nu<\delta_0$.
\item The (global) unstable and stable manifolds for $f_{\mu,\nu}$ of these saddle points are 
denoted by $W^{u}(P, f_{\mu,\nu})$, $W^{s}(P, f_{\mu,\nu})$ and so on. 
However, since the local unstable and stable manifolds in $U_{P}$ and $U_{Q}$ 
do not affected by the perturbations, these manifolds may be denoted by the notations same as the originals, e.g.,  
$W^{u}_{\loc}(P)$, $W^{s}_{\loc}(P)$ and so on.
\item For a small $\eta>0$, the 
$\eta$-neighborhoods of $Y$ in $W^u(Q)$ and $\tilde Y$ in $W^s(P)$ are denoted by $D_Y^u$ and 
$D_{\tilde Y}^s$ respectively.
\item
Since continuations of $D^{u}_{Y}$ in $W^u(Q,f_{\mu,\nu})$ vary with respect to $\mu$, $\nu$, 
they should be denoted by $D^{u}_{Y}(f_{\mu,\nu})$. 
\end{itemize}
\end{remark}

\begin{proof}[Proof of Proposition \ref{prop2.1}]
Let $\{f_{\mu,\nu}\}$ be the above $2$-parameter family with $f_{0,0}=f$.
The component $L$ of $W^{u}(P,f)\cap U_{X}$ containing $X$ is a segment not parallel to 
the $y$-axis of the linearizing coordinate of $U_Q$, see in Remark \ref{rmk1}-(i).
We move $L$ by the perturbation (\ref{perturb1}) with respect to $\nu$ and 
have a continuation $L_\nu$ of $L$ such that $L_{\nu}\subset W^{u}(P, f_{0,\nu})$
and  $L_{\nu}\to L$ as $\nu\to 0$. 
Here $L_{\nu}\to L$ means that $L_{\nu}$ $C^{2}$-converges to $L$.

For an arbitrarily small $\varepsilon>0$, take an integer $n_{0}>0$ 
satisfying  
$$
  |\beta_{c}|^{-n_{0}}<\varepsilon/2\quad\mbox{and}\quad K|\beta_{s}|^{n_{0}}<\varepsilon/2,
$$
where
$K=\|Df^{N_{2}}(\tilde{Y})\|$. 
Moreover, one can take the $n_{0}$ so that, 
for any $n>n_{0}$ and $\nu_{n}:=|\beta_{c}|^{-n}$, 
there exists a segment $\ell_{\nu_{n}}\subset L_{\nu_{n}}$ 
such that $f_{0,\nu_{n}}^{n}(\ell_{\nu_{n}})$ is a component of $f_{0,\nu_{n}}^{n}(L_{\nu_{n}})\cap U_Q$ 
satisfying
$$
\mathrm{dist}(f_{0,\nu_{n}}^{n}(\ell_{\nu_{n}}), \tilde{Y})<\varepsilon\quad\mbox{and}\quad  
\mathrm{angle}(
T f_{0,\nu_{n}}^{n}(\ell_{\nu_{n}}), 
T W_{\loc}^{u}(Q))<\varepsilon,
$$
where the distance and angle are defined with respect to the Euclidean metric on $U_{Q}$  
induced by the linearizing coordinate, see Figure \ref{fig2}.
We may assume that $D_Y^u(f_{0,\nu_n})\setminus \{Y\}$ is contained in the region of $U_P\setminus W_{\loc}^s(P)$ with $z< 0$.
 %%%%%%%%%%%%%%%%%%%%%%%%%%%%%%%%%%%%%%%
\begin{figure}[hbt]
\centering
\scalebox{0.59}{\includegraphics[clip]{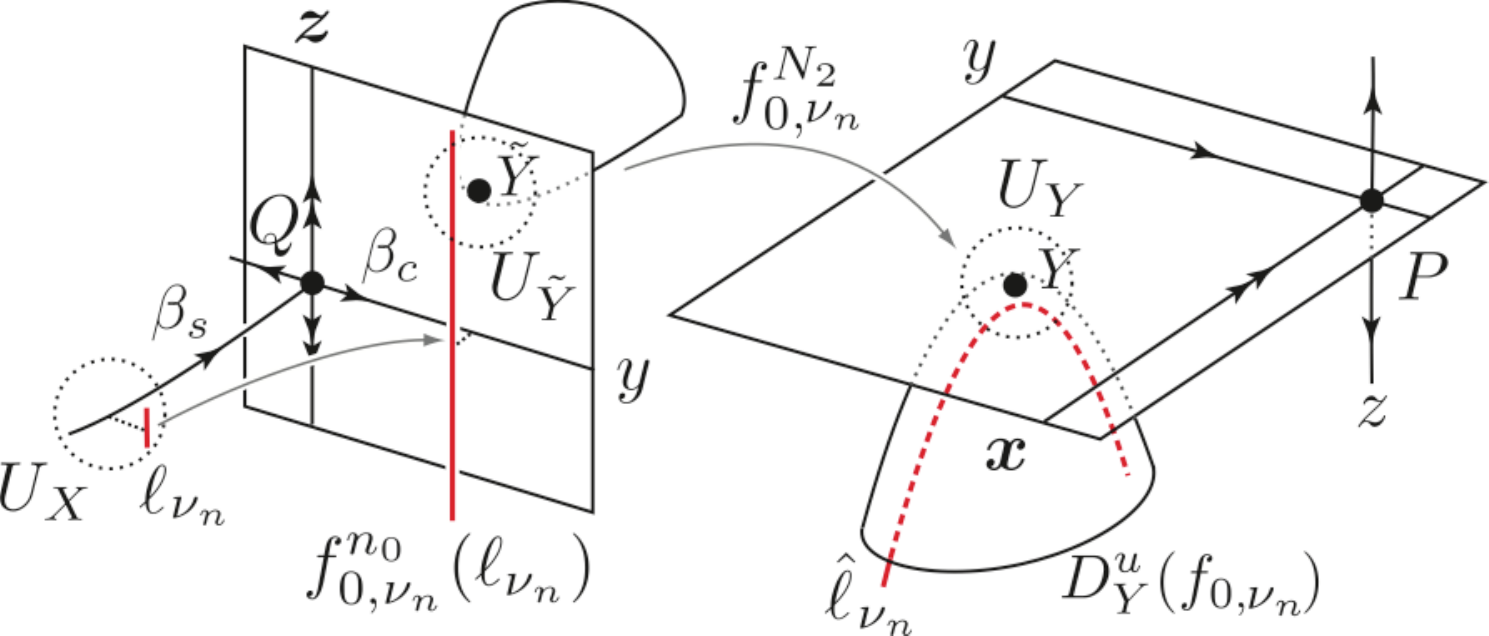}}
\caption{Transitions 
from $U_{X}$ to $U_{\tilde{Y}}$ and  $U_{\tilde{Y}}$ to $U_{Y}$.
}
\label{fig2}
\end{figure}
%%%%%%%%%%%%%%%%%%%%%%%%%%%%%%%%%%%%%%%

We set $\hat{\ell}_{\nu_{n}}
:=f_{0,\nu_{n}}^{N_{2}}\circ f_{0,\nu_{n}}^{n}(\ell_{\nu_{n}})$ 
and 
first consider the case that $\hat\ell_{\nu_{n}}$ is disjoint from $W_{\loc}^s(P)$ 
 as in Figure \ref{fig2}.
Then, we will adjust a value of the other parameter $\mu$ as follows so that the claim of this proposition holds.
Let $\hat{\ell}_{\mu,\nu_{n}}$ be 
a continuation of $\hat{\ell}_{\nu_{n}}$ such that $\hat{\ell}_{\mu,\nu_{n}}\subset W^{u}(P, f_{\mu,\nu_{n}})$ and  
$\hat{\ell}_{\mu,\nu_{n}}\to \hat{\ell}_{\nu_{n}}$ as $\mu\to 0$.
By the perturbation (\ref{perturb1}) with respect to $\mu$, any point in $U_{Y}$ close to $Y$ 
is just moved in parallel with the $z$-axis of $U_{P}$.
Thus, 
by the Intermediate Value Theorem,  
there exists a $\mu_{n}$ such that 
\begin{itemize}
\item
$0<|\mu_{n}|<  2K |\beta_{s}|^{n}$;
\item $\hat{\ell}_{\mu_{n},\nu_{n}}$ and  $W^{s}_{\loc}(P)$ have a quadratic tangency;
\item
$D^{u}_{Y}(f_{\mu_{n},\nu_{n}})$ meets $W^{s}_{\loc}(P)$ nontrivially and transversely. 
\end{itemize}
Moreover, since $Y$ is a tangency of elliptic type (Remark \ref{rmk1}-(ii)), 
$D^{u}_{Y}(f_{\mu_{n},\nu_{n}})\pitchfork W^{s}_{\loc}(P)$
is 
a  $(d-2)$-sphere in $W^{u}(Q, f_{\mu_{n},\nu_{n}})\cap W^{s}_{\loc}(P)$.

Next we consider the case that $\hat \ell_{\nu_n}$ intersects $W_{\loc}^s(P)$.
For $-1\leq a<b\leq 1$, let $A_{a,b}$ be the open subset of $U_P$ defined 
as $A_{a,b}=\{(\boldsymbol{x},y,z)\in U_P\,;\, a<z<b\}$.
Since $P$ is contained in the basic set $\Lambda_1$, for any $0<t<1$, there exists a continuation 
of subsurfaces $H_{\mu,\nu}$ of $W^s(P,f_{\mu,\nu})\cap U_P$ which are almost horizontal and contained in 
either $A_{0,t}$ or $A_{-t,0}$ for any $(\mu,\nu)$ sufficiently near $(0,0)$.
Since $\hat \ell_{\nu_n}$ converges to an arc in $D_Y^u$ as $n\to \infty$, 
$\hat \ell_{\nu_n}$ is contained in $A_{-1,t}$ and moreover disjoint from 
$H_{0,\nu_n}$ in the case of $H_{0,\nu_n}\subset A_{0,t}$ for all sufficiently large $n$.
When $H_{0,\nu_n}\subset A_{0,t}$, there exists $\mu_n$ with $0<\mu_n<t$ such that $\hat \ell_{\mu_n\nu_n}$ and $H_{\mu_n,\nu_n}$ have a quadratic tangency and $D^{u}_{Y}(f_{\mu_{n},\nu_{n}})
\pitchfork W_{\loc}^s(P)$ is a $(d-2)$-sphere, see Figure \ref{fig2_0} (a).
%%%%%%%%%%%%%%%%%%%%%%%%%%%%%%%%%%%%%%%
\begin{figure}[hbt]
\centering
\scalebox{0.65}{\includegraphics[clip]{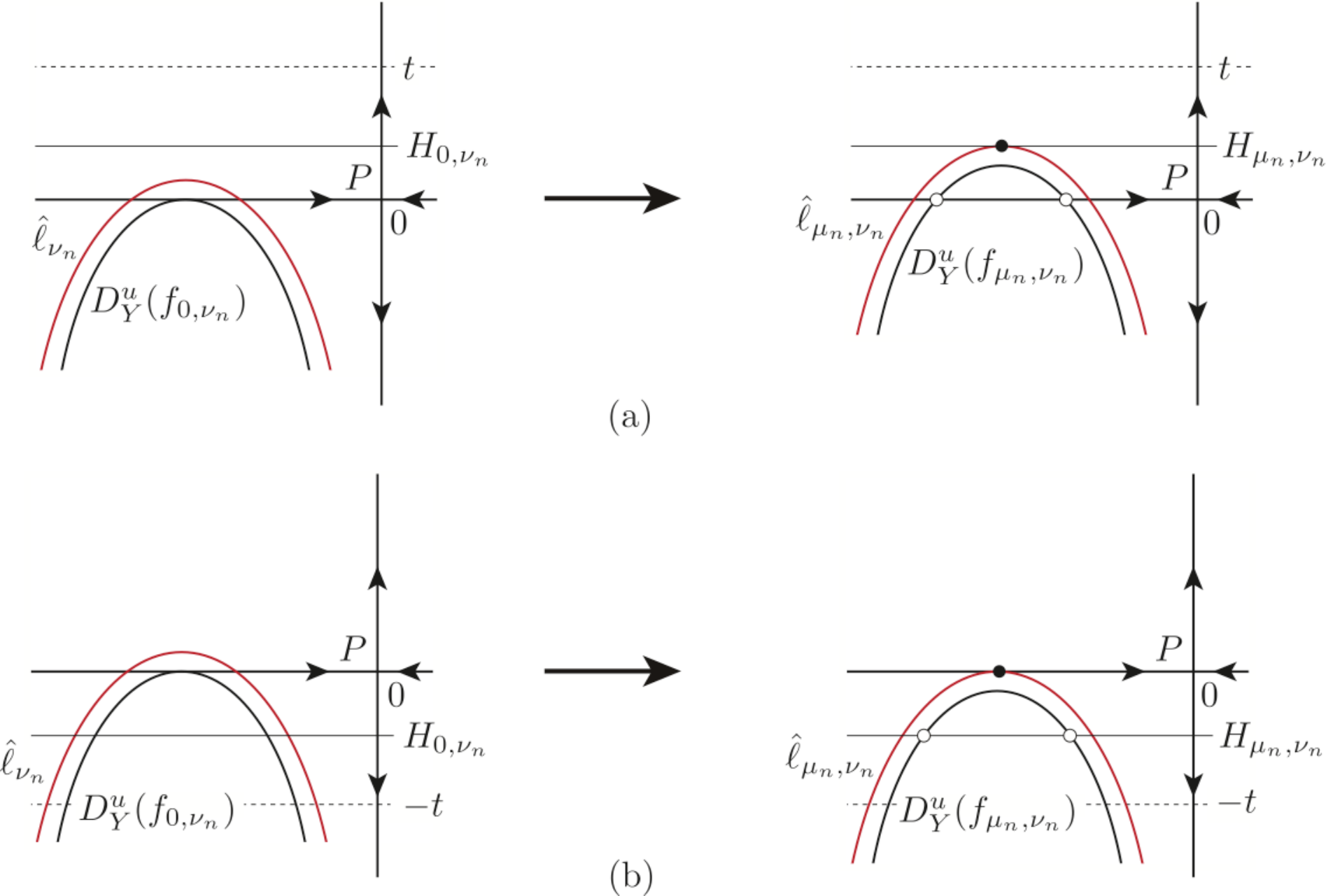}}
\caption{The black dots represent homoclinic tangencies and the pairs of white dots 
represent $(d-2)$-spheres.
}
\label{fig2_0}
\end{figure}
%%%%%%%%%%%%%%%%%%%%%%%%%%%%%%%%%%%%%%% 
When $H_{0,\nu_n}\subset A_{-t,0}$, one can choose $\mu_n$ with $-t<\mu_n<0$ so that $\hat \ell_{\mu_n\nu_n}$ and $W_{\loc}^p(P)$ have 
a quadratic tangency and $D^{u}_{Y}(f_{\mu_{n},\nu_{n}})\pitchfork H_{\mu_n,\nu_n}$ 
is a $(d-2)$-sphere, see Figure \ref{fig2_0} (b). 
Since $t$ can be taken arbitrarily small, we may choose $\mu_n$ so that  
$\lim_{n\to \infty}\mu_n=0$.

Thus, in ether case, $g_1:=f_{\mu_{n},\nu_{n}}$ satisfies our desired conditions.
This completes the proof.
\end{proof}

\section{Stable and unstable thicknesses}\label{S3}
In this section, we will recall the definition of the thickness given in Newhouse \cite{N79} for a 
Cantor set $K$ in $\mathbb{R}$.
Let $I$ be the minimal interval containing $K$.
A \textit{gap} of $K$ is a connected component of $I\setminus K$.
An ordering $\mathcal{G}=\{G_n\}$ of the gaps is called a \emph{presentation} of $K$.
For any $x\in \partial G_n$, the $\mathcal{G}$-component of $K$ at $x$ is the connected component $C$ of 
$I\setminus (G_1\cup\cdots\cup G_n)$ containing $x$.
For each such $x$, set $\tau(K,\mathcal{G},x)={\mathrm{Length}(C)}/{\mathrm{Length}(G_n)}$.
Then the \emph{thickness} of $K$ is given by
$$\tau(K)=\sup_{\mathcal{G}}\inf_x\tau(K,\mathcal{G},x),$$
where the infimum is taken over all boundary points of gaps of $K$.
The \emph{local thickness} of $K$ at $x\in K$ is defined as
$$\tau(K,x)=\limsup_{\varepsilon\to 0}\bigl\{\tau(L);\,L\subset K\cap [x-\varepsilon,x+\varepsilon] \text{ a Cantor set}\bigr\}.$$

The notion of thickness can be extended to that of a nontrivial basic set $\Lambda=\Lambda_{\varphi}$ of index $1$ as follows.
Let $z$ be a point of $W_{\loc}^{s}(\Lambda)$ and 
$\pi:I\to M$ a $C^{1}$ embedding transverse to $W_{\loc}^{s}(\Lambda)$ at $z =\pi(0)$, where $I$ is a closed interval containing $0$ as an interior point. 
Actually, $C^{1}$ projections along leaves of the stable $C^{1}$ foliation of $\Lambda$ 
can be used to define  $\pi$. 
The \emph{local stable thickness} of $\Lambda$ at $z$ is 
$\tau^{s}(\Lambda,z):=\tau(\pi^{-1}(W_{\loc}^{s}(\Lambda)), 0)$. 
Note that 
$\tau^{s}(\Lambda,z)$ is independent of the choice of $\pi$ 
and 
has the identical value for every $z\in W_{\loc}^{s}(\Lambda)$ which is a strictly positive finite number.
Thus we may denote it simply by $\tau^{s}(\Lambda)$
and call the (\emph{local}) \emph{stable thickness} of $\Lambda$.
On the other hand, it depends continuously on the diffeomorphism, that is, 
for any $\varphi_{1}$ sufficiently $C^{2}$-close to $\varphi_{2}$,  $|\tau^{s}(\Lambda_{\varphi_{1}})-\tau^{s}(\Lambda_{\varphi_{2}})|$ is smaller than a given positive constant,
see \cite[\S 4.3]{PT93}.
The \emph{local unstable thicknesses}  $\tau^{u}(\Lambda')$ of a basic set $\Lambda'$ with 
$\mathrm{index}(\Lambda')=d-1$ is defined similarly.

Let $\Lambda$, $\Gamma$ be nontrivial basic sets  with 
\begin{equation}\label{index-condition}
\mathrm{index}(\Lambda)=d-1,\quad \mathrm{index}(\Gamma)=1, 
\end{equation}
and let $\gamma:I\to M$ be a $C^1$ embedding transverse to $W^u(p)$ and $W^s(q)$ for some points $p\in \Lambda$ and $q\in \Gamma$.
Since $\Lambda$ is a basic set, there exists a Cantor set $K^u$ in some open subinterval $J$ of $I$ and an open segment $\alpha$ in $W^s(p,\Lambda)$ such that $(\gamma(J),\gamma(K^u))$ is $C^1$ diffeomorphic to 
$(\alpha,\Lambda\cap \alpha)$ along an unstable foliation associated with $W^u(\Lambda)$.
A Cantor set $K^s$ in $I$ is defined similarly from $\Gamma$.
We say that $K^u$ and $K^s$ are \emph{linked} if 
$K^s$ is not contained in a single component of $\mathbb{R}\setminus K^u$ and $K^u$ is not contained in a 
single component of $\mathbb{R}\setminus K^s$. 
By the Gap Lemma (\cite{N79,PT93}), if $\tau^u(\Lambda)\tau^s(\Gamma)>1$ and $K^u$ and $K^s$ are linked, then $K^u\cap K^s\neq \emptyset$.

Note that the above explanation could not be applied directly to 
 the basic sets without the condition (\ref{index-condition}). 
But, under the (codimension-one) sectionally dissipative condition,
one can define the stable thickness of $\Lambda$ with the same property as above
by using intrinsically $C^{1}$ projections along the leaves of an intrinsically $C^{1}$ 
foliation of $\Lambda$, see  \cite[\S 2-4]{PV94}  for details.

\section{Palis-Viana's setting and heterodimensional tangencies}\label{sec4}
In what follows, to simplify notations,
we denote continuations of the saddles $P$, $Q$ and the basic sets $\Lambda_0$, $\Lambda_1$ of every perturbed diffeomorphism  near  $f$ by the same notations as the original ones if it does not cause 
any confusion.

Let us start with the $C^{2}$ diffeomorphism $g_1$ which is arbitrarily $C^{2}$ close to
$f$ given in Proposition \ref{prop2.1} and satisfies the setting in Palis-Viana \cite{PV94}.  
In fact, 
$g_1$ has a quadratic homoclinic tangency in $U_{Y}$ 
associated with the saddle point $P$ satisfying the   
sectionally dissipative condition (\ref{P-eigenvalues}).  
By Palis-Viana \cite{PV94}, we have the following result.

\begin{proposition}[Steps 1-3 in {\cite[\S 7]{PV94}}]\label{prop4.1}
There exists a 
$g_{2}\in \Diff^{2}(M)$ arbitrarily $C^{2}$ close to $g_1$ which has  
a basic set $\Lambda_{2}$ of index $1$ other than $\Lambda_1$ such that 
\begin{enumerate}[\rm (1)]
\item 
$\tau^{u}(\Lambda_{1})\tau^{s}(\Lambda_{2})>1$;
\item $\Lambda_{1}$ and $\Lambda_{2}$ are homoclinically related;
\item
There are periodic points $P_{1}\in \Lambda_{1}$, $P_{2}\in \Lambda_{2}$ such that 
$W^{u}(P_{1},g_{2})$ and $W^{s}(P_{2},g_{2})$ have a quadratic tangency.\qed
\end{enumerate}
\end{proposition}

Actually, the above basic set $\Lambda_{2}$ 
is obtained by a renormalization 
of return maps defined on $U_{Y}$ which converge to H\'enon-like endomorphisms   
whose stable thickness of the invariant expanding Cantor set is arbitrarily large  
if $g_2$ is sufficiently close to $g_1=f_{\mu_n,\nu_n}$, see \cite[\S 6]{PV94}. 
Hence, one can suppose that
\begin{equation}\label{newhouse condition}
\tau^{u}(\Lambda_{0})\tau^{s}(\Lambda_{2})>1
\end{equation}
holds, where $\Lambda_{0}=\Lambda_{0,g_2}$ is the continuation of the original basic set $\Lambda_0$ in Theorem \ref{theoremA}.

Now we will show that there exists a diffeomorphism $g_3$ arbitrarily $C^2$ close to $g_2$ and such that 
 $W^{u}(\Lambda_{0}, g_{3})\cap W^{s}(\Lambda_{2}, g_{3})$ 
contains a heterodimensional tangency of elliptic type.

For a given $q\in \Lambda_{0}$, 
a compact subset $A$ of $W^u(q,g_2)$ is called an \emph{unstable cylinder} with foliation $\mathcal{F}$ 
if there exists a $C^1$ diffeomorphism $h:[0,1]\times S^{d-2}\to A$ with $\mathcal{F}=\{h(\{t\}\times S^{d-2})\,;\,
0\leq t\leq 1\}$.
For $p\in \Lambda_{1}$, we say that a sequence of unstable cylinders  $A_{n}$ in $W^{u}(q,g_{2})$ with foliations $\mathcal{F}_n$
$C^1$ \emph{converges} to an arc $\alpha$ in $W^{u}(p,g_{2})$  
if it satisfies the following conditions.
\begin{itemize}
\item
The diameter of each 
leaf of $\mathcal{F}_n$ is less than a constant $\varepsilon_n>0$ with $\lim_{n\to \infty}\varepsilon_n=0$.
\item
There exist $C^1$ sections $\sigma_n:[0,1]\to A_n\subset M$ meeting all leaves of $\mathcal{F}_n$ 
transversely and $C^1$ converging to 
an embedding $\sigma_\infty :[0,1]\to M$ with $\sigma_\infty([0,1])=\alpha$.
\end{itemize}

\begin{lemma}\label{lem4.2}
Let $g_{2}$ be the diffeomorphism in Proposition \ref{prop4.1}. 
For any periodic points $q\in \Lambda_{0}$, $p\in \Lambda_{1}$, 
there exists a sequence of foliated unstable cylinders $A^{u}_{n}$ in $W^{u}(q,g_{2})$
$C^1$ converging to an arc in $W^{u}_{\loc}(p, g_{2})$ containing $p$ as an interior point.
\end{lemma}
\begin{proof}
To make the proof clear, we first show the lemma for the fixed points
$Q\in \Lambda_{0}$, $P\in \Lambda_{1}$. 
Since $g_{2}$ is sufficiently $C^{2}$ close to $g_1$, 
Proposition \ref{prop2.1} implies that 
$W^{s}_{\loc}(P, g_{2})\pitchfork W^{u}(Q, g_{2})$ contains a $(d-2)$-sphere $S_0^{d-2}$.
Let $A_{0}^{u}$ be a closed tubular neighborhood of $S_0^{d-2}$ in $W^{u}(Q,g_2)$ with foliation $\mathcal{F}_0$ 
each leaf of which is the intersection of $A_0^u$ and the level surface $z=t$ with respect to the linearizing coordinate on $U_P$ for some $t$ with $|t|\leq \varepsilon$, where $\varepsilon$ is a small positive number.
See Figure \ref{fig41}.
%%%%%%%%%%%%%%%%%%%%%%%%%%%%%%%%%%%%%%%
\begin{figure}[hbt]
\centering
\scalebox{0.75}{\includegraphics[clip]{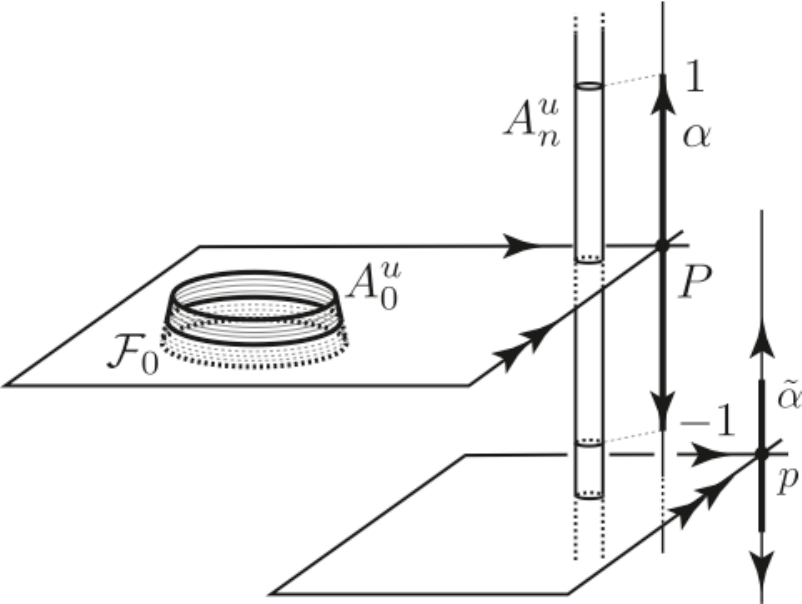}}
\caption{Foliated unstable cylinders.
}
\label{fig41}
\end{figure}
%%%%%%%%%%%%%%%%%%%%%%%%%%%%%%%%%%%%%%
For any sufficiently large $n$, there exists a sub-cylinder $A^u_n$ of the unstable cylinder 
$g_2^n(A_0^u)$ with $A_n^u\supset g_2^n(S_0^{d-2})$  
such that the boundary $\partial A_n^u$ is contained in the 
union of the level surfaces $z=\pm 1$.
The foliation $\mathcal{F}_n$ on $A_n^u$ is defined similarly as $\mathcal{F}_0$.
By (\ref{linearmap_P}), 
we conclude that the sequence of $A_{n}^{u}$ $C^1$ converges to the arc $\alpha=(\mathbf{0},0)\times [-1,1]$ in $W^{u}_{\loc}(P, g_{2})$.
Note that $P=(\mathbf{0},0)\times\{0\}$ is an interior point of $\alpha$.

Next, we observe that the assertion of  the lemma holds  for any periodic points $q\in \Lambda_{0}$ and $p\in \Lambda_{1}$. 
Since $W^{u}(q,g_{2})\pitchfork W^{s}(Q,g_{2})\neq \emptyset$ 
and
$W^{s}(p,g_{2})\pitchfork W^{u}(P,g_{2})\neq \emptyset$, 
by Inclination Lemma,  there exists 
a  $(d-1)$-dimensional disk in $W^{u}(q,g_{2})$ (resp.\ $W^{s}(p,g_{2})$)
which is arbitrarily $C^{1}$-close to $W^{u}_{\loc}(Q,g_{2})$ (resp.\ $W^{s}_{\loc}(P,g_{2})$).
Thus, $W^{s}(p, g_{2})\pitchfork W^{u}(q, g_{2})$ contains a $(d-2)$-sphere $\tilde{S}_0^{d-2}$ 
arbitrarily $C^{1}$-close to $S_0^{d-2}$. 
Moreover, one has a closed tubular neighborhood $\tilde{A}^{u}_{0}$ of $\tilde{S}_0^{d-2}$ 
in $W^{u}(q, g_{2})$ 
which is  arbitrarily $C^{1}$-close to $A^{u}_{0}$.
Since $W^{s}(p,g_{2})\pitchfork W^{u}(P,g_{2})\neq \emptyset$, 
the unstable cylinder $g_{2}^{m}(\tilde{A}^{u}_{0})$   
for sufficiently large $m$ and $W^{s}_{\loc}(p, g_{2})$ has 
a transverse $(d-2)$-sphere intersection. 
Thus, by a similar argument to the first case, one can obtain sub-cylinders 
$\tilde{A}^{u}_{n}$ of $g_{2}^{m+n\mathrm{per}(p)}(\tilde{A}^{u}_{0})$ which 
$C^{1}$-converges to the arc $\tilde\alpha$ containing $p$ in $W^{u}_{\loc}(p, g_{2})$. 
\end{proof}

\begin{proposition}\label{prop4.3}
There exists a 
$g_{3}\in \Diff^{2}(M)$ arbitrarily $C^{2}$ close to $g_{2}$ of Proposition \ref{prop4.1} with the 
saddle periodic points $P$, $Q$ and the 
continuations of the basic sets $\Lambda_{0}$, $\Lambda_{1}$, $\Lambda_{2}$  
such that 
\begin{enumerate}[\rm (1)]
\item
$\Lambda_{1}$ and $\Lambda_{2}$ are homoclinically related:
\item $Q\in \Lambda_{0}$, $P\in \Lambda_{1}$;
\item $W^{u}(\Lambda_{0}, g_{3})\cap W^{s}(\Lambda_{2}, g_{3})$ 
contains a heterodimensional tangency $r$ of elliptic type.
\end{enumerate}
\end{proposition}
\begin{proof}
The first and second claims are obtained immediately from (2) of Proposition \ref{prop4.1}.
By (3) of Proposition \ref{prop4.1}, there exist periodic points $P_{1}\in \Lambda_{1}$, $P_{2}\in \Lambda_{2}$ and an arc $L^{u}$ in $W^{u}(P_{1},g_{2})$ which has a quadratic tangency $R$ with $W^{s}_{\loc}(P_{2},g_{2})$.
If necessary replacing $L^u$ by a shorter arc, 
we may assume that $L^u$ is contained entirely
 in one component of $U_{R}\setminus W^{s}_{\loc}(P_{2},g_{2})$, 
 say in the lower component, where $U_{R}$ is the $\eta$-neighborhood  of $R$ in $M$ 
 with small $\eta>0$. 
By Lemma \ref{lem4.2}, there exists a sequence of unstable cylinders $A_n^u$ in $W^u(Q,g_2)$ $C^1$ converging to an arc $\alpha$ in $W_{\loc}^u(P_1,g_2)$ containing $P_1$ as an interior point.
Then we have a subarc $\alpha'$ of $\alpha$ and $m\in \mathbb{N}$ such that $g_2^m(\alpha')=L^u$.
Let ${A_n^u}'$ be sub-cylinders of $A_n^u$ $C^1$ converging to $\alpha'$.
Then ${A_n^u}''=g_2^m({A_n^u}')$ are unstable cylinders in $W^u(Q,g_2)$ $C^1$ converging to $L^u$, see Figure \ref{fig4.2}.
%%%%%%%%%%%%%%%%%%%%%%%%%%%%%%%%%%%%%%%
\begin{figure}[hbt]
\centering
\scalebox{0.78}{\includegraphics[clip]{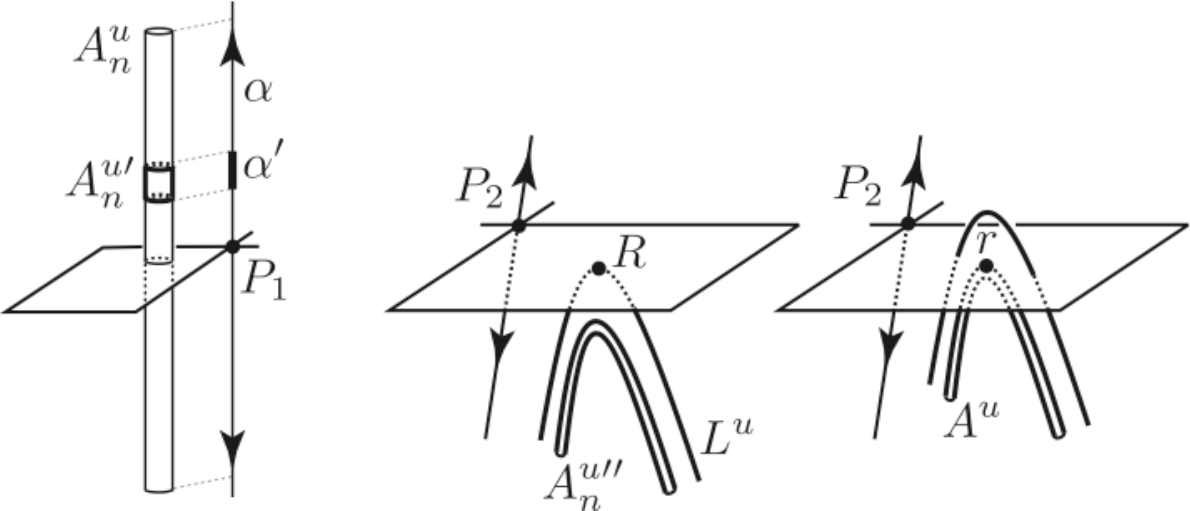}}
\caption{Creation of heterodimensional tangency.
}
\label{fig4.2}
\end{figure}
%%%%%%%%%%%%%%%%%%%%%%%%%%%%%%%%%%%%%%%

Note that the tangency between $L^{u}$ and $W^{s}(P_{2},g_{2})$
 unfolds generically with respect to the parameter of the H\'enon-like family 
given by 
the renormalization in Proposition \ref{prop4.1} as in \cite[\S6 and Step 3 of \S 7]{PV94}.
Thus, by controlling the parameter and applying 
the Intermediate Value Theorem, one can get a 
$g_{3}\in \Diff^{2}(M)$ arbitrarily  close to $g_{2}$
such that 
the continuation $A^u$ of ${A_n^{u}}''$ in $W^{u}(Q,g_{3})$ 
has a tangency $r$ with $W^{s}(P_{2},g_{3})$ and $A^u\setminus \{r\}$ lies in the lower component of $U_R\setminus W_{\loc}^s(P_2,g_3)$ if we take $n$ sufficiently large.
In particular, all the eigenvalues of the Hessian matrix of $A^u$ at $r$ relative to $W^s(P_2,g_3)$ are non-positive.
Sightly modifying $g_3$ by $C^2$ perturbation if necessary, we may suppose that all these eigenvalues are 
strictly negative.
It follows that the tangency $r$ is of elliptic type.
Thus the proof is complete.
\end{proof}

The assertion of Proposition \ref{prop4.3} (1) will be used to prove Theorem \ref{theoremB}, see 
Remark \ref{remark6.1}.

Consider the unstable manifold $W^u(q,g_3)$ associated with $q\in \Lambda_0$.
Let $\gamma:I=(-\varepsilon,\varepsilon)\to M$ be a short $C^1$ regular curve meeting $W^u(q,g_3)$ 
transversely at $\gamma(0)$.
Since $\Lambda_0$ is a nontrivial basic set of index $d-1$, 
there exists a Cantor set $K_{0,g_3}$ in $I$ with $K_{0,g_3}\ni 0$ which is defined as $K^u$ in Section \ref{S3}.
We say that $W^u(q,g_3)$ is \emph{two-sided} if $(-\eta,0)\cap K_{0,g_3}\neq \emptyset$ and $(0,\eta)\cap K_{0,g_3}\neq \emptyset$ for any $0<\eta<\varepsilon$.
The two-sided stable manifold $W^s(p,g_3)$ with $p\in \Lambda_2$ is defined similarly.
Since a Cantor set is perfect by definition, there exists $t\in K_{0,g_3}$ arbitrarily close to $0$ (possibly $t=0$) such that the unstable manifold $W^u(q',g_3)$ for a $q'\in \Lambda_0$ with 
$W^u(q',g_3)\ni \gamma(t)$ is two-sided.

If necessary modifying the diffeomorphism $g_3$ given in Proposition \ref{prop4.3} slightly, one can suppose that 
$g_3$ satisfies the following.

\begin{corollary}\label{cor4.4}
The unstable manifolds $W^u(q,g_3)$ in $W^u(\Lambda_0,g_3)$ and the stable manifold $W^s(p,g_3)$ in $W^s(\Lambda_2,g_3)$ containing the heterodimensional tangency $r$  
are two-sided.
\end{corollary}
\begin{proof}
Any $C^1$ regular curve in $M$ meeting $W^u(q,g_3)$ with $q\in \Lambda_{0}$ transversely at $r$ meets also $W^s(p,g_3)$ with $p\in \Lambda_{2}$
transversely at $r$.
By using the curve, one can show that there exist $q'\in \Lambda_0$ and $p'\in \Lambda_2$ such that both $W^u(q',g_3)$ and $W^s(p',g_3)$ are 
two-sided and intersect arbitrarily small neighborhood $U$ of $r$ in $M$.
Thus, we have a $g_3'\in \Diff^2(M)$ arbitrarily $C^2$ close to $g_3$ such that $W^u(q',g_3')$ and $W^s(p',g_3')$ have a heterodimensional 
tangency $r'$ of elliptic type in $U$.
The proof is completed by setting $p'=p$, $q'=q$, $g_3'=g_3$, $r'=r$ again.
\end{proof}

\section{Arc of tangencies of stable and unstable foliations}
Recall that $M$ is a closed $C^2$ manifold with $\dim M=d\geq 3$.
We denote the diffeomorphism $g_{3}$ obtained in Proposition \ref{prop4.3} by $g$ for simplicity.
Thus, $g\in \Diff^{2}(M)$ has nontrivial basic sets $\Lambda_0$ and $\Lambda_2$ with
$\index(\Lambda_0)=d-1$ and $\index(\Lambda_2)=1$.
Moreover, 
$W^{u}(\Lambda_{0}, g)$ and $W^{s}(\Lambda_{2},g)$ have a heterodimensional tangency.
Let ${\mathcal F}$ and $\widetilde{\mathcal F}$ be stable and unstable foliations associated with 
$\Lambda_0$ and $\Lambda_2$. 
Note that $W_{\mathrm{loc}}^u(\Lambda_0)$ and $W_{\mathrm{loc}}^s(\Lambda_2)$ are considered to be sublaminations of ${\mathcal F}$ and $\widetilde{\mathcal F}$ respectively.
Both ${\mathcal F}$, $\widetilde{\mathcal F}$ are $C^1$ foliations on $M$, but each leaf of ${\mathcal F}$ or $\widetilde{\mathcal F}$ is a codimension-one $C^2$ submanifold of $M$.  
The aim of this section is to show that, under suitable conditions, there exists a regular $C^1$ curve $\gamma:(-\varepsilon,\varepsilon)\to M$ which meets leaves of both ${\mathcal F}$ and $\widetilde{\mathcal F}$ transversely and such that, for any $t\in (-\varepsilon,\varepsilon)$, $\gamma(t)$ is a non-degenerate heterodimensional 
tangency between leaves of  ${\mathcal F}$ and  $\widetilde{\mathcal F}$.

Suppose that there exist leaves $\lambda_0$ of ${\mathcal F}$ and $\tilde \lambda_0$ of $\widetilde{\mathcal F}$ which have a strict tangency at a point $p$ in $M$.
A small open neighborhood $U$ of $p$ in $M$ has a $C^2$ coordinate $(\boldsymbol{x},z)$ with $p=(\mathbf{0},0)$  
such that the leaf $\lambda_0$ is contained in the level surface $z=0$, where 
$\boldsymbol{x}=(x_1,\dots,x_{d-1})$, $\mathbf{0}=(0,\dots,0)\in \mathbb{R}^{d-1}$. 
This gives the identification of $U$ with an open neighborhood of $(\mathbf{0},0)$ in $\mathbb{R}^d$.

Since ${\mathcal F}$ is a $C^1$ foliation on $M$ with $C^2$ leaves, for a sufficiently small $\delta>0$, there exists a $C^1$ diffeomorphism
\begin{equation}\label{eqn_varphi}
\varphi_g:(-\delta,\delta)^d=(-\delta,\delta)^{d-1}\times (-\delta,\delta)\to U
\end{equation}
with $\varphi_g(\boldsymbol{x},z)=(\boldsymbol{x},\alpha(\boldsymbol{x},z))$ and satisfying the following conditions, where we set $\varphi_g=\varphi$ for short.
\begin{itemize}
\item
$\alpha(\boldsymbol{x},z)$ is a $C^2$ map on $\boldsymbol{x}$ and $C^1$ on $z$.
\item
$\varphi(\boldsymbol{x},0)=(\boldsymbol{x},0)$ for any $\boldsymbol{x}\in (-\delta,\delta)^{d-1}$.
\item
For any $z\in (-\delta,\delta)$, $\varphi((-\delta,\delta)^{d-1}\times \{z\})$ is contained in a 
leaf of $\mathcal F$.
\end{itemize}

For any $(\boldsymbol{x},z)\in (-\delta,\delta)^d$, let 
$$\widetilde{\boldsymbol{N}}_{(\boldsymbol{x},z)}=\bigl(\boldsymbol{\mu}(\boldsymbol{x},z),\nu(\boldsymbol{x},z)\bigr)=\bigl(\mu_1(\boldsymbol{x},z),\dots,\mu_{d-1}(\boldsymbol{x},z),\nu(\boldsymbol{x},z)\bigr)$$
be the unit tangent vector of 
$T_{\varphi(\boldsymbol{x},z)}(U)$ orthogonal to the leaf of $\widetilde{\mathcal{F}}$ containing $\varphi(\boldsymbol{x},z)$ and such that the $d$-th entry $\nu(\boldsymbol{x},z)$ of $\widetilde{\boldsymbol{N}}_{(\boldsymbol{x},z)}$ is positive, see Figure \ref{fig5.1}.
%%%%%%%%%%%%%%%%%%%%%%%%%%%%%%%%%%%%%%%
\begin{figure}[hbt]
\centering
\scalebox{0.75}{\includegraphics[clip]{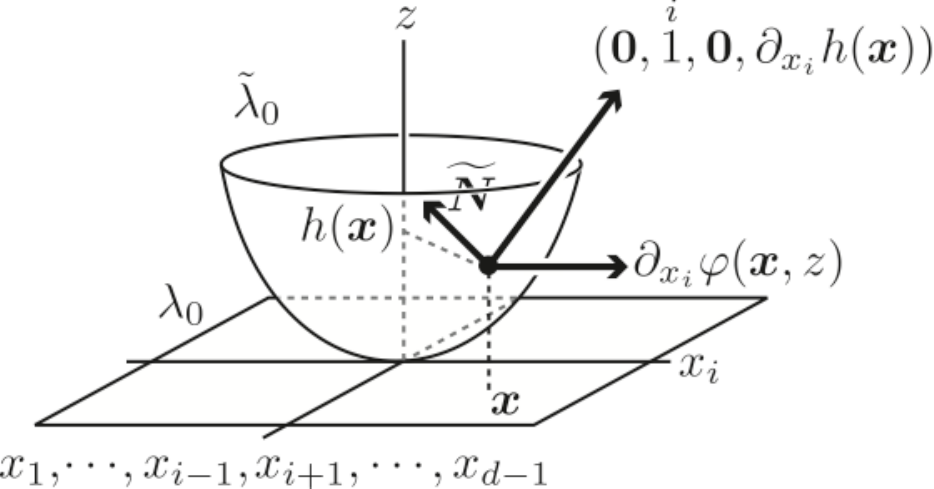}}
\caption{$h(\boldsymbol{x})=\alpha(\boldsymbol{x},z)$ if $\varphi(\boldsymbol{x},z)\in \tilde \lambda_0$.
}
\label{fig5.1}
\end{figure}
%%%%%%%%%%%%%%%%%%%%%%%%%%%%%%%%%%%%%%%.
Consider the $C^1$ map $\psi_g=\psi:(-\delta,\delta)^d\to \mathbb{R}^{d-1}$ defined as 
$$\psi(\boldsymbol{x},z)=(\widetilde{\boldsymbol{N}}_{(\boldsymbol{x},z)}\cdot \partial_{x_1}\varphi(\boldsymbol{x},z),\dots,
\widetilde{\boldsymbol{N}}_{(\boldsymbol{x},z)}\cdot \partial_{x_{d-1}}\varphi(\boldsymbol{x},z))$$
Then the Jacobian matrix of $\psi(\boldsymbol{x},z)$ is 
\begin{equation*}
J\psi(\boldsymbol{x},z)=
\begin{pmatrix}
a_{1,1}(\boldsymbol{x},z)&\cdots
&a_{1,d-1}(\boldsymbol{x},z)
&b_1(\boldsymbol{x},z)\\
\vdots&\ddots&\vdots&\vdots\\
a_{d-1,1}(\boldsymbol{x},z)&\cdots
&a_{d-1,d-1}(\boldsymbol{x},z)
&b_{d-1}(\boldsymbol{x},z)
\end{pmatrix},
\end{equation*}
where
\begin{align*}
a_{i,j}(\boldsymbol{x},z)&=\partial_{x_j}\widetilde{\boldsymbol{N}}_{(\boldsymbol{x},z)}\cdot\partial_{x_i} \varphi(\boldsymbol{x},z)+\widetilde{\boldsymbol{N}}_{(\boldsymbol{x},z)}\cdot
\partial_{x_ix_j}\varphi(\boldsymbol{x},z),\\
b_i(\boldsymbol{x},z)&=
\partial_z\widetilde{\boldsymbol{N}}_{(\boldsymbol{x},z)}\cdot\partial_{x_i} \varphi(\boldsymbol{x},z)+\widetilde{\boldsymbol{N}}_{(\boldsymbol{x},z)}\cdot
\partial_{x_iz}\varphi(\boldsymbol{x},z).
\end{align*}
Since 
$\widetilde{\boldsymbol{N}}(\mathbf{0},0)=(\mathbf{0},1)$,
$
\partial_{x_i}\varphi(\mathbf{0},0)=(0,\dots,0,\overset{i}{1},0,\dots,0)$, 
$
\partial_{x_ix_j}\varphi(\mathbf{0},0)=(\mathbf{0},0)$ and 
$
\partial_{x_iz}\varphi(\mathbf{0},0)= (\mathbf{0},   \partial_{x_iz}\alpha(\mathbf{0}^{d}) )
$, 
we have
{\small \begin{equation}\label{eqn_0}
J\psi(\mathbf{0}^d)=
\begin{pmatrix}
\partial_{x_1}\mu_1(\mathbf{0}^d)&\cdots
&\partial_{x_{d-1}}\mu_1(\mathbf{0}^d)
&\partial_z\mu_1(\mathbf{0}^d)+\partial_{x_1z}\alpha(\mathbf{0}^d)\\
\vdots&\ddots&\vdots&\vdots\\
\partial_{x_1}\mu_{d-1}(\mathbf{0}^d)&\cdots
&\partial_{x_{d-1}}\mu_{d-1}(\mathbf{0}^d)
&\partial_y\mu_{d-1}(\mathbf{0}^d)+\partial_{x_{d-1}z}\alpha(\mathbf{0}^d)
\end{pmatrix},
\end{equation}
}
where $\mathbf{0}^d=(\mathbf{0},0)\in \mathbb{R}^d$.

Let $h:(-\delta,\delta)^{d-1}\to \mathbb{R}$ be a $C^2$ function such that the graph 
$\{(\boldsymbol{x},h(\boldsymbol{x}))\,;\, \boldsymbol{x}\in (-\delta,\delta)^{d-1}\}$ of $h$ is contained in the leaf $\tilde\lambda_0$ of 
$\widetilde{\mathcal F}$.
For any $\boldsymbol{x}\in (-\delta,\delta)^{d-1}$ and $i=1,\dots,d-1$, the vector $(0,\dots,0,\overset{i}{1},0,\dots,0,\partial_{x_i} h(\boldsymbol{x},0))$ is tangent to 
$\tilde \lambda_0$ at $(\boldsymbol{x},h(\boldsymbol{x}))$, see Figure \ref{fig5.1}.
It follows that
$$
\widetilde{\boldsymbol{N}}_{(\boldsymbol{x},h(\boldsymbol{x}))}\cdot (0,\dots,0,\overset{i}{1},0,\dots,0,\partial_{x_i} h(\boldsymbol{x}))=\mu_i(\boldsymbol{x},h(\boldsymbol{x}))+\nu(\boldsymbol{x},h(\boldsymbol{x}))\partial_{x_i} h(\boldsymbol{x})=0.
$$
Differentiating the latter equation by $x_j$ $(j=1,\dots,d-1)$, we have
\begin{multline*}
\partial_{x_j}\mu_i(\boldsymbol{x},h(\boldsymbol{x}))+\partial_z\mu_i(\boldsymbol{x},h(\boldsymbol{x}))\partial_{x_j} h(\boldsymbol{x})\\
\begin{split}
=-\bigl(\partial_{x_j}\nu(\boldsymbol{x},h(\boldsymbol{x}))&+\partial_z\nu(\boldsymbol{x},h(\boldsymbol{x}))\partial_{x_j} h(\boldsymbol{x})\bigr)\partial_{x_j} h(\boldsymbol{x})\\
&-\nu(\boldsymbol{x},h(\boldsymbol{x}))\partial_{x_jx_i} h(\boldsymbol{x}).
\end{split}
\end{multline*}
Since $h(\mathbf{0})=0$, $\partial_{x_j} h(\mathbf{0})=0$ and $\nu(\mathbf{0},0)=1$,
\begin{equation}\label{eqn_1}
\partial_{x_j}\mu_i(\mathbf{0},0)=-\partial_{x_jx_i}h(\mathbf{0}).
\end{equation}

\begin{proposition}\label{prop5.1}
With the notation as above, the following two conditions are equivalent.
\begin{enumerate}[\rm (1)]
\item
The strict tangency of $\lambda_0$ and $\tilde\lambda_0$ at $p=(\mathbf{0},0)$ in $U$ is non-degenerate.
\item
There exists a $C^1$ regular curve $\gamma_g=\gamma:(-\varepsilon,\varepsilon)\to U$ for a 
sufficiently small $\varepsilon>0$ with $\gamma(0)=p$ and such that, for each $t\in (-\varepsilon,\varepsilon)$, the curve $\gamma(-\varepsilon,\varepsilon)$ meets leaves $\lambda_t$ of $\mathcal F$ and $\tilde\lambda_t$ of $\widetilde{\mathcal F}$ at $\gamma(t)$ transversely.
Moreover, $\lambda_t$ and $\tilde\lambda_t$ have a non-degenerate strict tangency at $\gamma(t)$.
\end{enumerate}
\end{proposition}

\begin{proof}
Since ``(2) $\Rightarrow$ (1)'' is obvious, we prove ``(1) $\Rightarrow$ (2)''.

Suppose that the tangency of $\lambda_0$ and $\tilde\lambda_0$ at $p$ is non-degenerate, or 
equivalently that
\begin{equation}\label{eqn_3}
\det\left(\partial_{x_ix_j}h(\mathbf{0})\right)\neq 0.
\end{equation}
Let $\varPsi_g=\varPsi:(-\delta,\delta)^d\to \mathbb{R}^d$ be the $C^1$ map defined as
\begin{equation}\label{eqn_4}
\begin{split}
\varPsi(\boldsymbol{x},z)&=(\psi(\boldsymbol{x},z),z)\\
&=\bigl(\widetilde{\boldsymbol{N}}_{(\boldsymbol{x},z)}\cdot \partial_{x_1}\varphi(\boldsymbol{x},z),\dots,
\widetilde{\boldsymbol{N}}_{(\boldsymbol{x},z)}\cdot \partial_{x_{d-1}}\varphi(\boldsymbol{x},z),\, z\bigr).
\end{split}
\end{equation}
By (\ref{eqn_0}) and (\ref{eqn_1}), the Jacobian matrix of $\varPsi$ at $(\mathbf{0},0)$ 
has the form
$$
J\varPsi(\mathbf{0},0)=
\begin{pmatrix}
-\partial_{x_1x_1}h(\mathbf{0})&\cdots
&-\partial_{x_1x_{d-1}}h(\mathbf{0})
&\partial_z\mu_{1}(\mathbf{0},0)+\partial_{x_1z}\alpha(\mathbf{0},0)\\
\vdots&\ddots&\vdots&\vdots\\
-\partial_{x_1x_{d-1}}h(\mathbf{0})&\cdots
&-\partial_{x_{d-1}x_{d-1}}h(\mathbf{0})
&\partial_z\mu_{d-1}(\mathbf{0},0)+\partial_{x_{d-1}z}\alpha(\mathbf{0},0)\\
0&\cdots&0&1
\end{pmatrix}.
$$
Then (\ref{eqn_3}) implies that $J\varPsi(\mathbf{0},0)$ is regular.
By the Inverse Function Theorem, there exists a $C^1$ local inverse  $\varPsi^{-1}:(-\varepsilon,\varepsilon)^d\to (-\delta,\delta)^d$ of $\varPsi$ with $\varPsi^{-1}(\mathbf{0},0)=(\mathbf{0},0)$ for 
a sufficiently small $\varepsilon>0$.
Consider the $C^1$ map $\gamma:(-\varepsilon,\varepsilon)\to U$ defined by $\gamma(t)=\varphi\circ \varPsi^{-1}(\mathbf{0},t)$.
Since the $d$-th entry of $\partial_z\varPsi^{-1}(\mathbf{0},t)$ is one, the map $\varPsi^{-1}(\mathbf{0},t)$ of $t$ defines 
a regular curve in $(-\delta,\delta)^d$ passing through the $z$-constant level surfaces transversely.
This implies that $\gamma(-\varepsilon,\varepsilon)$ is a regular curve in $U$ transverse to leaves of $\mathcal{F}$.
Let $\lambda_t$ (resp.\ $\tilde\lambda_t$) be the leaf of $\mathcal F$ (resp.\ $\widetilde{\mathcal F}$) 
containing $\gamma(t)$.
From the definition (\ref{eqn_4}) of $\varPsi$, we know that the normal vector $\widetilde{\boldsymbol{N}}(\varPsi^{-1}(\mathbf{0},t))$ of  $\tilde\lambda_t$ at $\gamma(t)$ 
is orthogonal to the tangent vectors $\partial_{x_i}\varphi(\varPsi^{-1}(\mathbf{0},t))$ 
$(i=1,\dots,d-1)$ of $\lambda_t$.
This shows that $\lambda_t$ is tangent strictly to $\tilde\lambda_t$ at $\gamma(t)$.
In particular, the curve $\gamma(-\varepsilon,\varepsilon)$ meets $\tilde\lambda_t$ transversely at $\gamma(t)$ as well as it does $\lambda_t$.

From the continuity of the matrix $J\varPsi$, one can assume that the strict tangency of $\lambda_t$ and $\tilde\lambda_t$ is non-degenerate for any $t\in (-\varepsilon,\varepsilon)$ if necessary 
replacing $\varepsilon$ by a smaller positive constant.
This completes the proof. 
\end{proof}

\begin{corollary}\label{cor5.2}
If the condition {\rm (1)} in Proposition \ref{prop5.1} holds, then there is an open neighborhood 
$\mathcal{O}_g$ of $g$ in $\Diff^2(M)$ such that, for any $\tilde g$ in $\mathcal{O}_g$, there exists a 
$C^1$ regular curve $\gamma_{\tilde g}:(-\varepsilon,\varepsilon)\to M$ satisfying the condition 
on $\tilde g$ corresponding to {\rm (2)} of Proposition \ref{prop5.1} and depending on 
$\tilde g\in \mathcal{O}_g$ continuously.
\end{corollary}
\begin{proof}
According to Propositions 1 and 2 in Pollicott \cite{Pol03}, there exists a small open neighborhood $\mathcal{O}_g$ of $g$ in $\Diff^2(M)$ such that, for any $\tilde g\in \mathcal{O}_g$, there are stable and unstable foliations associated respectively with $\Lambda_{0,\tilde g}$ and $\Lambda_{2,\tilde g}$ which $C^0$ vary on $\mathcal{O}_g$.
It follows that the $C^1$ map $\varPsi_{\tilde g}$ defined as (\ref{eqn_4}) 
and its inverse $\varPsi_{\tilde g}^{-1}$ depend on $\tilde g\in \mathcal{O}_g$ continuously.
Thus the $C^1$ regular curve $\gamma_{\tilde g}:(-\varepsilon,\varepsilon)\to U$ defined by 
$\gamma_{\tilde g}(t)=\varphi_{\tilde g}\circ \varPsi_{\tilde g}^{-1}(\mathbf{0},t)$ satisfies 
the condition (2) of Proposition \ref{prop5.1} and depends continuously on $\tilde g\in \mathcal{O}_g$, 
where $\varphi_{\tilde g}:(-\varepsilon,\varepsilon)^d\to U$ is a diffeomorphism defined from $\tilde g$ 
as (\ref{eqn_varphi}).
\end{proof}

\section{Proof of Theorem \ref{theoremA}}\label{S6}

\begin{proof}[Proof of Theorem \ref{theoremA}]
By combining results which have been obtained before this section, 
we have a sequence $\{f_n\}$ in $\Diff^2(M)$ $C^2$ converging to $f$ and satisfying the following 
conditions (i)-(iv)
\begin{enumerate}[(i)]
\item
By Proposition \ref{prop4.3}, each $f_n$ has saddle periodic points $Q_{f_n}$, $P_{f_n}$ and nontrivial 
basic sets $\Lambda_{0,f_n}$, $\Lambda_{1,f_n}$, $\Lambda_{2,f_n}$ such that $Q_{f_n}\in \Lambda_{0,f_n}$, $P_{f_n}\in 
\Lambda_{1,f_n}$, $\mathrm{index}(\Lambda_{0,f_n})=d-1$, $\mathrm{index}(\Lambda_{1,f_n})=\mathrm{index}(\Lambda_{2,f_n})=1$, and $\Lambda_{1,f_n}$ and $\Lambda_{2,f_n}$ are homoclinically related, where $\Lambda_{0,f_n}$ is the continuation of the basic set $\Lambda_{0}=
\Lambda$ in Theorem \ref{theoremA}.
Moreover, $W^u(\Lambda_{0,f_n},f_n)$ and $W^s(\Lambda_{2,f_n},f_n)$ have a heterodimensional tangency $r_n$ of elliptic type.
\item
By Corollary \ref{cor4.4}, 
we may assume that both the leaves of $W^u(\Lambda_{0,f_n},f_n)$ and $W^s(\Lambda_{2,f_n},f_n)$ containing $r_n$ are 
two-sided.
\item
By Proposition \ref{prop5.1} and Corollary \ref{cor5.2}, there is an open neighborhood $\mathcal{O}_n$ of $f_n$ in $\Diff^2(M)$ such that, for any $g\in \mathcal{O}_n$, there exists a $C^1$ regular curve $\gamma_g:(-\varepsilon,\varepsilon)\to M$ which satisfies the condition 
corresponding to (2) of Proposition \ref{prop5.1} and depends continuously on 
$g\in \mathcal{O}_n$.
Furthermore, in the case of $g=f_n$, the curve $\gamma_{f_n}$ satisfies $\gamma_{f_n}(0)=r_n$.
\item
By (\ref{newhouse condition}) together with the local continuity of thickness in the $C^2$ topology, 
%\begin{equation}\label{tau^u tau^s}
$$
\tau^u(\Lambda_{0,g})\tau^s(\Lambda_{2,g})>1
$$
%\end{equation}
for any $g\in \mathcal{O}_n$, where $\Lambda_{0,g}$, $\Lambda_{2,g}$ are the continuations of 
$\Lambda_{0,f_n}$, $\Lambda_{2,f_n}$ respectively.
\end{enumerate}

\medskip

Let $K_{0,g}$ and $K_{2,g}$ be Cantor sets defined as $K^u$  and $K^s$ in Section \ref{S3}.
By (ii) and (iii), we may assume that $K_{0,g}$ and $K_{2,g}$ are linked for any $g\in \mathcal{O}_n$ if 
necessary replacing $\mathcal{O}_n$ by a smaller open neighborhood of $f_n$.
Since $\tau(K_{0,g})\tau(K_{2,g})>1$ by (iv), the Gap Lemma (\cite{N79,PT93}) implies that 
$K_{0,g}\cap K_{2,g}\neq \emptyset$ for any $g\in \mathcal{O}_n$.
For any $t\in K_{0,g}\cap K_{2,g}$, $\gamma_g(t)$ is a heterodimensional tangency of elliptic type 
associated with $\Lambda_{0,g}$ and $\Lambda_{2,g}$.
Thus the union $\mathcal{O}=\bigcup_{n=1}^\infty \mathcal{O}_n$ is an open set of $\Diff^2(M)$ whose 
closure contains $f$ and such that each $g\in \mathcal{O}$ has a heterodimensional tangency of elliptic type 
associated with the basic sets $\Lambda_{0,g}$ and $\Lambda_{2,g}$.
This completes the proof of Theorem \ref{theoremA}.
\end{proof}

\begin{remark}\label{remark6.1}
By Proposition \ref{prop4.3} (1), one can choose the open set $\mathcal{O}_n$ in the proof of 
Theorem \ref{theoremA} so that, for 
any $g\in \mathcal{O}_n$, $\Lambda_{1,g}$ is homoclinically related to $\Lambda_{2,g}$.
\end{remark}

\section{Proof of Theorem \ref{theoremB}}\label{S7}
To prepare for the proof of Theorem \ref{theoremB}, 
let us  introduce several definitions in the $C^{2}$ topology.
Let $f\in \Diff^{2}(M)$ with $\dim M=d\geq 3$. 
A nontrivial basic set $\Lambda_{0}$ of $\index(\Lambda_0)=d-1$   
is a \emph{cu-blender} for $f$ if 
there exist a neighborhood $\mathcal{U}_{f}\subset \Diff^{2}(M)$ of $f$  and 
a $C^{1}$ open set $\mathcal{D}$ of  
$(d-2)$-dimensional disks $D$ embedded in $M$ 
 such that for every $g\in \mathcal{U}_{f}$ and every $D\in \mathcal{D}$, 
 $$
 W^{s}_{\loc}(\Lambda_{0,g},g)\cap D\neq \emptyset.
 $$
The set $\mathcal{D}$  is called the \emph{superposition region} of 
the cu-blender. See Definition 3.1 in \cite{BD12, BDK12}. 

In fact, the cu-blenders considered in previous works of  Bonatti-D\'iaz 
belong to a special class of blenders, called  \emph{blender-horseshoes}, see
\cite[\S 1]{BD96}, \cite[Definition 3.8]{BD12}.
In these works, the blender-horseshoe $\Lambda_{0}$ 
is the maximal invariant set of $f$ in a $d$-dimensional `cube' $\Delta$ in $M$ 
and has a hyperbolic splitting with three nontrivial bundles 
$T_{\Lambda_{0}}M=E^{s}\oplus E^{cu}\oplus E^{uu}$ 
where $E^{s}$ is the stable bundle with $\dim E^{s}=1$  
and $E^{u}=E^{cu}\oplus E^{uu}$ is  the unstable bundle with  $\dim E^{cu}=1$.
Moreover, there exists an integer  $k>0$ such that 
$f^{k}\vert \Lambda_{0}$ is topologically conjugate to the complete shift of two symbols. 
In practice of \cite[\S 1]{BD96}, 
each of the bundles is obtained as a 
limit of stable, unstable, strong unstable cones 
 $\mathcal{C}^s$, $\mathcal{C}^{uu}$, $\mathcal{C}^u$ defined on $\Delta$,
 respectively. In particular,  
$$E^{s}\subset \mathcal{C}^s,\ 
E^{uu}\subset \mathcal{C}^{uu}\subset \mathcal{C}^u,\  
E^{cu}\oplus E^{uu}\subset \mathcal{C}^{u}.$$
Note that the construction of the blender-horseshoe 
implies that $\Delta\cap f^{k}(\Delta)$  consists of two components 
each of which contains a \emph{distinguished} saddle periodic point in $\Lambda_{0}$.
We may suppose that  
one of them is  
a fixed  point, say $Q$,  
while the other is 
$k$-periodic point,  say $Q^{\prime}$.  See Figure \ref{fig6_1}.

We consider \emph{vertical disks} $D$ through 
the blender-horseshoe, 
that is, each $D$ is a $(d-2)$-dimensional  disk tangent to $\mathcal{C}^{uu}$
and joining the `top' and the `bottom' of the cube $\Delta$.
Then there are two isotopy classes of vertical disks that do not intersect 
$W^{s}_{\loc}(Q, f)$ (respectively $W^{s}_{\loc}(Q^{\prime}, f^{k})$), called 
disks at the right and at the left of $W^{s}_{\loc}(Q, f)$ (resp.\ $W^{s}_{\loc}(Q^{\prime}, f^{k})$). 
For example,  $W^{uu}_{\loc}(Q, f)$ (that is a vertical disk) is at the left of 
$W^{s}_{\loc}(Q^{\prime}, f^{k})$ as in Figure \ref{fig6_1}. Similarly, 
$W^{uu}_{\loc}(Q^{\prime}, f^{k})$  is at the right of 
$W^{s}_{\loc}(Q, f)$. Note that 
the superposition region $\mathcal{D}$ of the  blender-horseshoe 
consists of the vertical disks in between $W^{s}_{\loc}(Q, f)$ and 
$W^{s}_{\loc}(Q^{\prime}, f^{k})$. 
See for more details in \cite{BD96,BDVbook,BD12, BDK12}.

%%%%%%%%%%%%%%%%%%%%%%%%%%%%%%%%%%%%%%%
\begin{figure}[hbt]
\centering
\scalebox{0.82}{\includegraphics[clip]{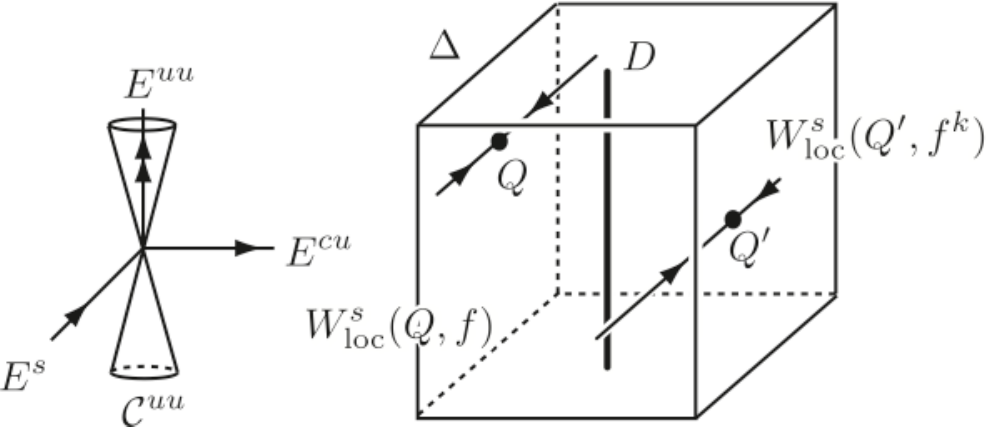}}
\caption{Vertical disks for a blender-horseshoe.
}
\label{fig6_1}
\end{figure}
%%%%%%%%%%%%%%%%%%%%%%%%%%%%%%%%%%%%%%%

\begin{proof}[Proof of Theorem \ref{theoremB}]
Let $f$ be a  $d$-dimensional $C^{2}$ diffeomorphism with 
the nontrivial basic sets $\Lambda_0$, $\Lambda_1$
satisfying the same conditions as in (1), (2) of Theorem \ref{theoremA}.
Moreover, we 
suppose that the $\Lambda_0$  is a blender-horseshoe of $\index(\Lambda_0)=d-1$, 
which is the maximal invariant set of a $d$-dimensional cube $\Delta$.
From (1), $P\in \Lambda_1$ is of index one. 
As to the unstable manifold $W^{u}(P,f)$ of $P$, 
we furthermore suppose  that 
there exist 
a segment $L^u_{f}$ of $W^{u}(P,f)$, 
a constant $\delta_{0}>0$ and 
a $C^{2}$ embedding  $D: [-1,1] \times[-\delta_{0}, \delta_{0}]^{d-3} \to \Delta$ 
such that 
\begin{itemize}
\item $D([-1,1] \times \{\boldsymbol{0}^{d-3}\})=L^u_{f}$;
\item for any $0<\delta<\delta_{0}$, 
$D_\delta (L_f^u):=D([-1,1]  \times[-\delta, \delta]^{d-3})$ 
is a vertical disk through the blender-horseshoe $\Lambda_{0}$ 
and contained in the superposition region $\mathcal{D}$.
\end{itemize}

Let $\mathcal{O}$ be the open set obtained in Theorem \ref{theoremA}.
Then the closure of $\mathcal{O}\cap \mathcal{U}_{f}$ contains $f$.
By Theorem \ref{theoremA}, for every $g\in \mathcal{O}\cap \mathcal{U}_{f}$, 
$W^{u}(\Lambda_{0,g},g)$ 
and 
$W^{s}(\Lambda_{2,g},g)$  
have a heterodimensional tangency of elliptic type.
Note that the blender-horseshoe is 
robust for any small $C^r$ perturbation with $r \geq 1$, see \cite[Remark 3.2]{BDK12}.
Hence, one has 
a disk $D_{\delta}(L^u_{g})$ in $\Delta$ with $L^u_{g}\subset W^{u}(P_{g},g)$ such that  
$D_{\delta}(L^u_{g})\to D_{\delta}(L^u_{f})$
as $g\to f$.
Thus, we may assume that $D_{\delta}(L^u_{g})$
still belongs to  $\mathcal{D}$.
Hence, in particular, 
$$
W^{s}_{\loc}(\Lambda_{0,g},g)\cap D_{\delta}(L^u_{g})\neq\emptyset.
$$

By Theorem \ref{theoremA} and Remark \ref{remark6.1}, the basic set $\Lambda_{2,g}$ is nontrivial and homoclinically related to $\Lambda_{1,g}$.
Therefore, there exists a segment $\tilde L_g^u$ in $W^u(\tilde P,g)$ for some $\tilde P\in \Lambda_{2,g}$ 
such that 
$\tilde L_g^u$ is arbitrarily $C^{1}$ close to $L_g^u$ and 
$D_{\delta}(\tilde L_g^u)$  is also contained in  $\mathcal{D}$.
Thus,
$$
W^{s}_{\loc}(\Lambda_{0,g},g)\cap D_{\delta}(\tilde L_g^u)\neq\emptyset.
$$ 
Since $W^{s}_{\loc}(\Lambda_{0,g},g)$ is closed subset of $M$,  
$W^{s}_{\loc}(\Lambda_{0,g},g)$ and $\tilde L_g^u$
have non-empty intersection.
It follows that 
$$W^{s}(\Lambda_{0,g},g)\cap W^{u}(\Lambda_{2,g},g)\neq \emptyset.$$
This completes the proof of Theorem \ref{theoremB}.
\end{proof}

\section{Proof of Theorem \ref{theoremC}}\label{S8}

Finally, let us prove Theorem \ref{theoremC}. For that purpose, 
we first consider a $C^{2}$ diffeomorphism $f$ having nontrivial basic sets 
$\Lambda_{0}$, $\Lambda_{1}$ with the following $C^{2}$ open conditions.
\begin{enumerate}[(i)]
\item Each periodic point in $\Lambda_{0}$, $\Lambda_{1}$ satisfies the same condition as in (1) of Theorem \ref{theoremA}. 
\item
$f$ has a spherical heterodimensional intersection on the heterodimensional cycle
associated with $\Lambda_{0}$ and $\Lambda_{1}$.
\item
$\Lambda_{0}$ is a  blender-horseshoe.
\end{enumerate}
Let $\mathcal{D}$ be the superposition region of $\Lambda_{0}$. 
We denote the union  $\cup_{D\in \mathcal{D}} D$ 
by  $|\mathcal{D}|$.
As was mentioned above, 
$\Delta\cap f^{k}(\Delta)$ consists of two components.
Note that $|\tilde{\mathcal{D}}|:=|\mathcal{D}|\setminus  (\Delta\cap f^{k}(\Delta))$ is disjoint from 
$\Lambda_{0,f}$.
We say that a segment $L$ is \emph{in superposition} in  $|\tilde{\mathcal{D}}|$ if 
\begin{itemize}
\item $L\subset |\tilde{\mathcal{D}}|$ and  $TL\subset \mathcal{C}^{uu}$;
\item there exists a constant $\delta_{0}>0$ and 
a $C^{2}$ embedding  $D: [-1,1] \times[-\delta_{0}, \delta_{0}]^{d-3} \to \Delta$ 
 for any $0<\delta<\delta_{0}$ which satisfies 
 $$
 D([-1,1] \times\{\boldsymbol{0}^{d-3}\})=L,\quad 
D([-1,1] \times[-\delta, \delta]^{d-3})\in \mathcal{D}.
 $$
\end{itemize}
In addition to (i)--(iii), 
we suppose the following condition:
\begin{enumerate}[(iv)]
\item
There exist  a saddle periodic point $P$ in $\Lambda_{1}$ 
such that the unstable manifold $W^{u}(P,f)$ contains  segments 
$\ell_f^u, L_f^u$  in superposition in $|\tilde{\mathcal{D}}|$ 
whose orbits $\mathcal{O}(\ell_f^u,f)=\bigcup_{n\in\mathbb{Z}} f^{n}(\ell_f^u)$
and $\mathcal{O}(L_f^u,f)= \bigcup_{n\in\mathbb{Z}} f^{n}(L_f^u)$ are disjoint.
\end{enumerate}
Note that there exists an open subset   of $\Diff^{2}(M)$ 
whose element satisfies the condition (iv).  
In fact,  
one can construct an open set of examples 
from a certain diffeomorphism $f_{0}\in \Diff^{2}(M)$ with
 a \emph{partially hyperbolic saddle-node} periodic point $S$
with at least tree disjoint orbits of  \emph{strong homoclinic intersections}, 
i.e., 
\begin{itemize}
\item $Df_{0}^{\mathrm{per}(S)}(S)$ has eigenvalues 
$\alpha$, $\gamma_{1},\ldots,\gamma_{d-2}$, $\beta$ with 
$|\alpha|<1<|\beta|$, $|\gamma_{1}|=\cdots=|\gamma_{d-2}|=1$ and such that 
$\alpha$, $\beta$ are real and at least one of $\{\gamma_{i}\}$ is $1$.
\item 
Consider the strong unstable and stable invariant directions $E^{ss}$, $E^{uu}$ respectively corresponding to the eigenvalues $\alpha$, $\beta$  of $Df_{0}^{\mathrm{per}(S)}(S)$. 
The strong unstable manifold 
$W^{uu}(S,f_{0})$ of $S$ is the unique 
$f_{0}$-invariant manifold tangent to $E^{uu}$ of the same dimension as $E^{uu}$. 
The strong stable manifold $W^{ss}(S,f_{0})$ of $S$ is defined similarly by using $E^{ss}$ 
instead of $E^{uu}$.
 \item 
 $W^{ss}(S,f_{0})\cap W^{uu}(S,f_{0})$ contains at least three different orbits  which do not belong to the orbit of $S$. See Figure \ref{fig6_2}-(a).
\end{itemize}
Let us explain how such examples are constructed.
For simplicity,  we assume that $d=3$, 
$S$ is the saddle-node fixed point  and 
$W^{ss}(S,f_{0})\cap W^{uu}(S,f_{0})\setminus\{S\}$ 
contains three points
$X$, $Y$, $Z$ any one of which is not contained in the orbit of the other ones.
Suppose moreover that these points are quasi-transverse intersections associated with 
 $W^{ss}(S, f_{0})$ and  $W^{uu}(S, f_{0})$. 
After a small $C^{2}$ perturbation of $f_{0}$, 
we can have a diffeomorphism $f$ such that the  saddle-node fixed point $S$ splits into two hyperbolic fixed points 
$Q$ (expanding in the central direction) and $P$ (contracting in the central direction).
The saddle points $Q$ and $P$  have different indices and $W^{s}(P)$ and $W^{u}(Q)$ have a transverse intersection that contains the interior of a `central' curve jointing $Q$ and $P$. 
Moreover, from  \cite[\S 3.3]{BDK12}, 
$f$ has the following properties.
\begin{itemize}
\item There exists a $3$-dimensional cube $\Delta$ and an integer $k>0$ such that 
$\Delta \cap f^{k}(\Delta)$ consists of two disjoint components 
which respectively contain  $Q$ and 
a point $Z_{f}\in W^{u}(P,f)$ converging to $Z$ as $f\to f_{0}$. 
\item The maximal invariant set in $\Delta$ is 
a blender-horseshoe $\Lambda_{f}$ with distinguished fixed point $Q$ 
and $k$-periodic point $Q^{\prime}$. 
\end{itemize} 
Note that 
it has the superposition region $\mathcal{D}$ 
between $W^{s}_{\loc}(Q, f)$ and $W^{s}_{\loc}(Q^{\prime}, f^{k})$. 
%%%%%%%%%%%%%%%%%%%%%%%%%%%%%%%%%%%%%%%
\begin{figure}[hbt]
\centering
\scalebox{0.82}{\includegraphics[clip]{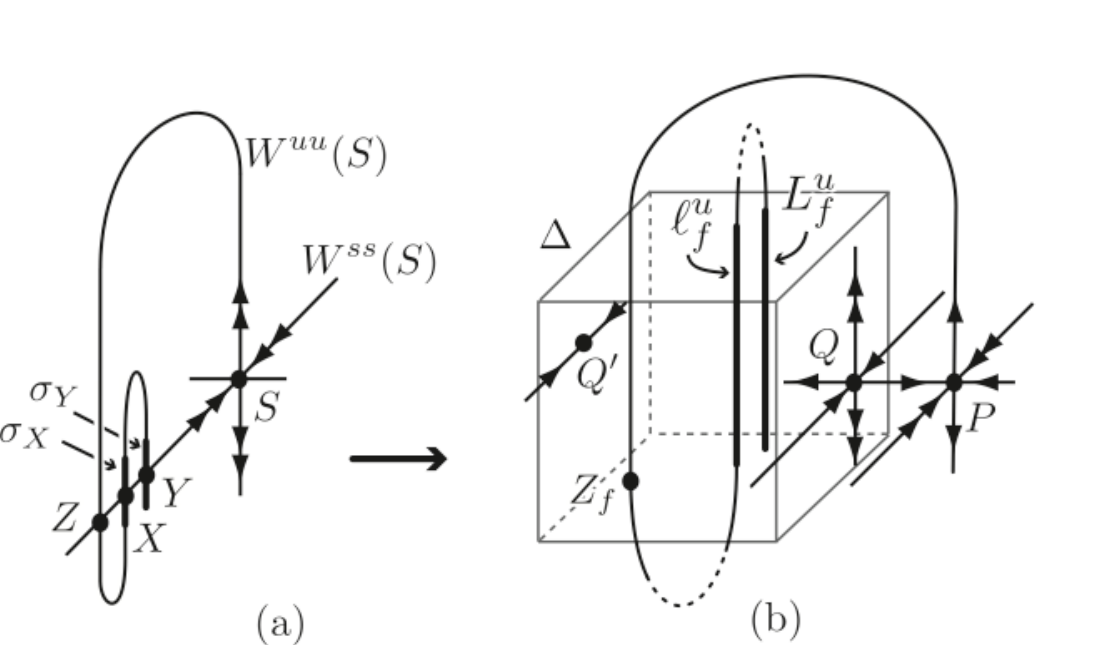}}
\caption{(a) 
Strong homoclinic intersections associated with $S$.
 (b) $\ell^{u}_{f}$ and  $L^{u}_{f}$ 
in $\tilde{\mathcal{D}}$.
}
\label{fig6_2}
\end{figure}
%%%%%%%%%%%%%%%%%%%%%%%%%%%%%%%%%%%%%%%

Since  the orbits $\mathcal{O}(X, f_{0})$ of $X$ and  $\mathcal{O}(Y, f_{0})$ of $Y$ are disjoint, 
one has small segments 
$\sigma_{X}, \sigma_{Y} \subset W^{uu}(S, f_{0})$  
with $X\in \mathrm{Int}\sigma_{X}$, $Y\in  \mathrm{Int}\sigma_{Y}$ such that 
$$\mathcal{O}(\sigma_{X}, f_{0})\cap  \mathcal{O}(\sigma_{Y},f_{0})=\emptyset.$$
Thus, for any $f$ sufficiently $C^{2}$-close to $f_{0}$, 
\begin{equation}\label{eqn_1}
\mathcal{O}(\sigma_{X_{f}}, f)\cap  \mathcal{O}(\sigma_{Y_{f}},f)=\emptyset
\end{equation}
where $\sigma_{X_f}, \sigma_{Y_f} \subset W^{uu}(S, f)$ are segments such that $\sigma_{X_{f}}\to \sigma_{X}$ and  $\sigma_{Y_{f}}\to \sigma_{Y}$ as $f\to f_{0}$. 
We may assume that 
$T_{X}\sigma_{X}$ and $T_{Y}\sigma_{Y}$ are not equal to the central direction of $S$ if 
necessary slightly $C^2$-perturbing $f$.
Hence, 
for sufficiently large integer $m>0$, we obtain
segments  $\ell_f^u\subset f^{m}(\sigma_{X_{f}})$  
$L_f^u\subset f^{m}(\sigma_{Y_{f}})$ which are  in superposition in $|\tilde{\mathcal{D}}|$.
Moreover, by (\ref{eqn_1}), we have $\mathcal{O}(\ell_f^u, f)\cap  \mathcal{O}(L_f^u,f)=\emptyset$.
Observe that 
the above property is open and corresponds to  (iv).

\begin{proof}[Proof of Theorem \ref{theoremC}] 
To prove this theorem, 
we have only to show that an arbitrarily small $C^{2}$ neighborhood of the above $f$ with (i)--(iv)   
contains a diffeomorphism $g$ having a heterodimensional tangency on a heterodimensional cycle associated with 
saddle periodic points which satisfy (1) and (2) of Theorem \ref{theoremA}.

By the conditions (i)--(ii), we have points  
$q_{1}\in \Lambda_{0,f}$  and $p_{1}\in \Lambda_{1, f}$ such that 
$W^{u}(q_{1},f)\pitchfork W^{s}(p_{1},f)$ contains  
a $(d-2)$-dimensional sphere $S^{d-1}_{f}$. See Figure \ref{fig6_3}. 
Moreover, by (iii)-(iv), 
there exist points $q_2\in \Lambda_{0,f}$, $p_{2}\in \Lambda_{1, f}$ and disjoint segments
 $\ell_f^u, L_{f}^{u}\in W^{u}(p_2,f)$ such that $\ell_f^u, L_{f}^{u}$ are  in 
 superposition $|\tilde{\mathcal{D}}|$   and 
$W_{\loc}^s(q_2,f)\cap \ell_f^u\neq \emptyset$.
Note that, in general,
$q_{1}, p_{1}, q_{2}$ are not periodic points, or 
$W_{\loc}^s(q_2,f)$ does not intersect  with $L_{f}^{u}$.
Since $T\ell_{f}^{u}\subset \mathcal{C}^{uu}$ and $TW^{s}_{\loc}(q_{2},f)\subset  \mathcal{C}^{s}$,
$W_{\loc}^s(q_2,f)\cap \ell_f^u$ consists of a 
single quasi-transverse intersection $X$, i.e., 
$T_{X}\ell_{f}^{u}+T_{X}W^{s}_{\loc}(q_{2},f)=T_{X}\ell_{f}^{u}\oplus T_{X}W^{s}_{\loc}(q_{2},f)$.

%%%%%%%%%%%%%%%%%%%%%%%%%%%%%%%%%%%%%%%
\begin{figure}[hbt]
\centering
\scalebox{0.8}{\includegraphics[clip]{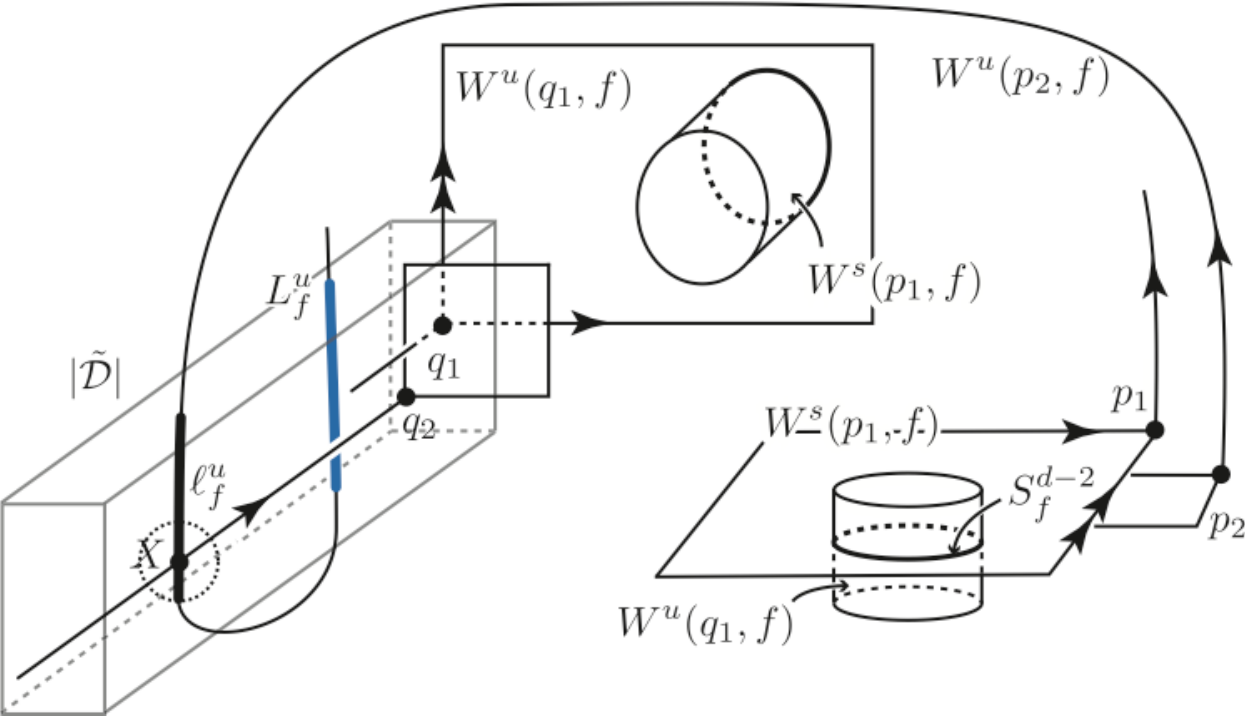}}
\caption{Configuration of primary items of $f$.
}
\label{fig6_3}
\end{figure}
%%%%%%%%%%%%%%%%%%%%%%%%%%%%%%%%%%%%%%%

To obtain our desired diffeomorphism $g$, we will add $C^{2}$ local perturbations to $f$ three times.
\smallskip

\noindent
\textbf{(I) The 1st perturbation from $f$ to $\tilde f$.}
We here consider a local $C^{2}$ perturbation of $f$ in a small neighborhood of $X$ 
obtained by a similar method as in Section \ref{sec2} as below.
For any arbitrarily small $\varepsilon>0$, 
let $(x_{1},\ldots, x_{d})$ be a local coordinate on the $\varepsilon$-neighborhood $U_{\varepsilon}$ 
of $X$ 
such that the $(x_{1},x_{2})$-plane corresponds to $T_{X}\ell_{f}^{u}+T_{X}W^{s}_{\loc}(q_{2},f)$ 
and the $(x_{3},\ldots, x_{d})$-space  is the orthogonal complement.
Consider a $C^{2}$ bump function 
$
B(x_{1},\ldots, x_{d})=\prod_{i=1}^{d} b(x_{i})
$
where $b$ is a  $C^{2}$ function  on $\mathbb{R}$ 
satisfying 
$$
\begin{cases}
b(x)= 0\ & \mbox{if}\  \varepsilon \leq \vert x\vert; \\
0< b(x) <1\ & \mbox{if}\  \varepsilon/2<|x|<\varepsilon;\\
b(x)= 1\  &  \mbox{if}\  \vert x\vert  \leq  \varepsilon/2.
\end{cases} 
$$
Fix $\varepsilon_0>0$ sufficiently smaller than $\varepsilon$.
For any $\mu\in (-\varepsilon_0, \varepsilon_0)$,  let $\varphi_{\mu}$ be a family of perturbations in $\Diff^{2}(M)$ such that 
if $(x_{1},\ldots, x_{d})\in U_{\varepsilon}$, 
$$ 
\varphi_{\mu}(x_{1},\ldots, x_{d})=(x_{1}, x_{2}, 0,\ldots,0)+
\mu B(x_{1},\ldots, x_{d})(0,0,1,\ldots,1),
$$
otherwise, $\varphi_{\mu}$ is the identity. 

Let us define $f_{\mu}:=\varphi_{\mu}\circ f$.
Since $U_{\varepsilon}$ is contained   in  $\tilde{|\mathcal{D}|}$, 
the above perturbation will not affect $\Lambda_{0}$, that is, $\Lambda_{0,f_{\mu}}=\Lambda_{0}$ for every $\mu$. 
Moreover, observe that $W^{s}_{\loc}(\Lambda_{0}, f_{\mu})=W^{s}_{\loc}(\Lambda_{0}, f)$ for every 
$\mu\in (-\varepsilon_0, \varepsilon_0)$, 
while $\ell_{f_{\mu}}^{u}$ and $W^{s}_{\loc}(q_{2},  f_{\mu})$  no longer meet in $U_{\varepsilon}$ 
if $\mu\neq 0$.
However, since periodic points are dense in the basic set $\Lambda_{0}$, 
by the Intermediate Value Theorem, 
one can obtain $\tilde\mu \in (-\varepsilon_0, \varepsilon_0)\setminus\{0\}$
such that 
the diffeomorphism $\tilde f:=f_{\tilde\mu}$  
satisfies the following conditions.
\begin{itemize}
\item There are  a periodic point
$\tilde{q}_{2}\in  \Lambda_{0}$
arbitrarily near $q_{2}$, a periodic point
$\tilde{p}_{2}\in \Lambda_{1}$  
arbitrarily near $p_{2}$ 
and segments $\ell_{\tilde{f}}^{u}$, $L_{\tilde{f}}^{u}$ of
$W^{u}(\tilde{p}_{2}, \tilde{f})$ with $\ell_{\tilde{f}}^{u}\to \ell_{f}^{u}$, $L^u_{\tilde{f}}\to L^{u}_{f}$ as $\tilde{f}\to f$ such that 
$\ell_{\tilde{f}}^{u}$, $L_{\tilde{f}}^{u}$ are  in $|\tilde{\mathcal{D}}|$, 
$\ell_{\tilde{f}}^{u}$ and  $W_{\loc}^{s}(\tilde{q}_{2}, \tilde{f})$ have a quasi-transverse intersection $\tilde X$. See Figure \ref{fig6_4}.
\end{itemize}
\smallskip

%%%%%%%%%%%%%%%%%%%%%%%%%%%%%%%%%%%%%%%
\begin{figure}[hbt]
\centering
\scalebox{0.77}{\includegraphics[clip]{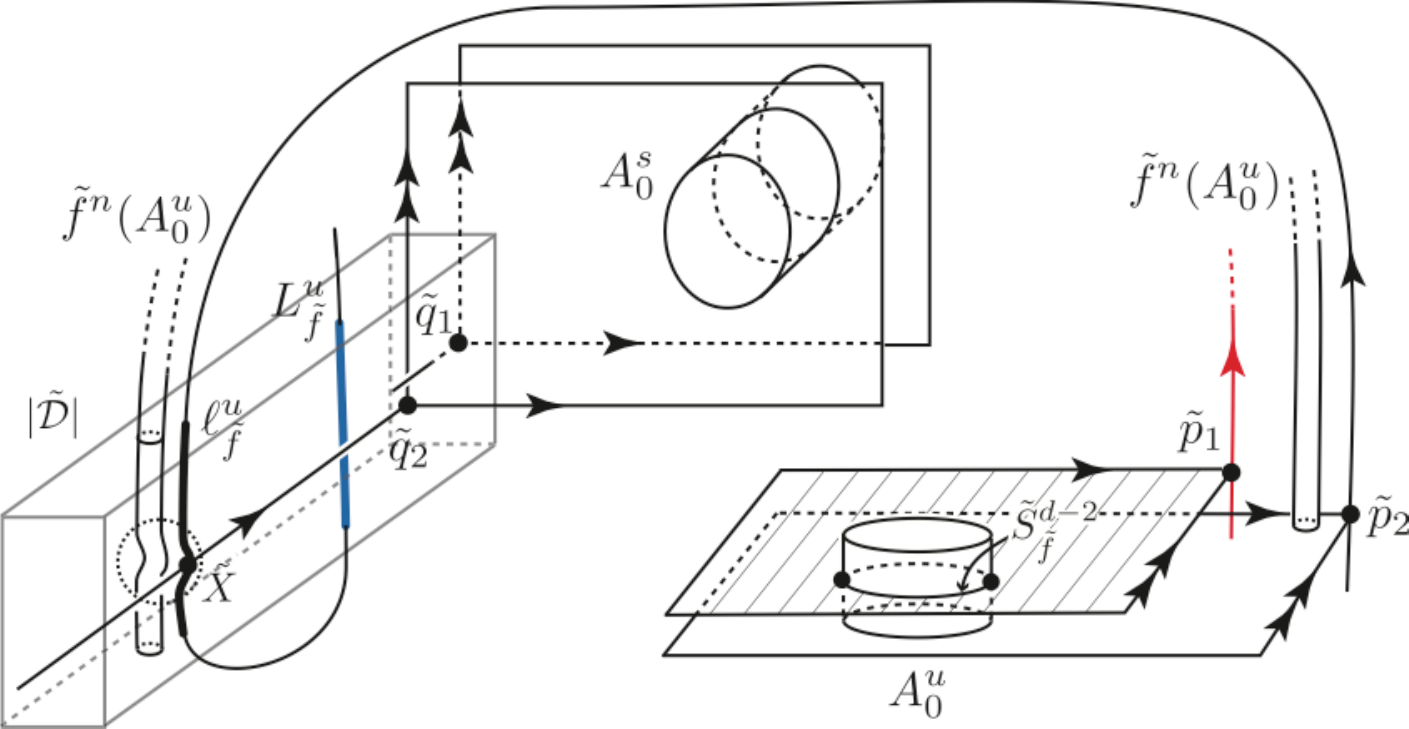}}
\caption{$\tilde{q}_{1}, \tilde{q}_{2} \in  \Lambda_{0}$ and $\tilde{p}_{1}, \tilde{p}_{2}\in   \Lambda_{1}$.
}
\label{fig6_4}
\end{figure}
%%%%%%%%%%%%%%%%%%%%%%%%%%%%%%%%%%%%%%%

\noindent
\textbf{(II) The 2nd perturbation from $\tilde f$ to $\tilde g$.}
On the other hand, since the $(d-2)$-dimensional sphere $S^{d-2}_{f}$ is contained in 
the transverse intersection $W^{u}(q_{1},f)\pitchfork W^{s}(p_{1},f)$ and 
periodic points are dense in the basic sets $\Lambda_{0}$, $\Lambda_{1}$, one has the following property:
\begin{itemize}
\item 
for every $\tilde{f}$ sufficiently $C^{2}$-close to $f$, there exists a periodic point
$\tilde{q}_{1}\in  \Lambda_{0}$
arbitrarily near $q_{1}$ and a periodic point
$\tilde{p}_{1}\in \Lambda_{1}$  
arbitrarily near $p_{1}$ such that 
$W^{u}(\tilde{q}_{1}, \tilde{f})\pitchfork W^{s}(\tilde{p}_{1}, \tilde{f})$ contains 
a $(d-2)$-sphere $\tilde{S}^{d-2}_{\tilde f}$  
which has at least two tangencies  with 
leaves of $\mathcal{F}^{ss}(\tilde{p}_{1},\tilde{f})$. See Figure \ref{fig6_4}.
\end{itemize}

Since $f|_{\Lambda_1}=\tilde f|_{\Lambda_1}$ and $\Lambda_1$ is sectionally dissipative by (i), $\tilde p_1$ is a  sectionally dissipative periodic point of $\tilde f$. 
Observe that there exists a foliated unstable cylinder $A_0^u$ in $W^{u}(\tilde{q}_{1}, \tilde{f})$ 
and a foliated stable cylinder $A_0^s$ in $W^{s}(\tilde{p}_{1}, \tilde{f})$ such that 
$$ A_0^u\cap W^{s}(\tilde{p}_{1},\tilde{f})=
\tilde f^N(  A_0^s \cap W^{u}(\tilde{q}_{1},\tilde{f}))
=\tilde{S}^{d-2}_{\tilde f}$$ 
for some integer $N>0$.
As in the proofs of Lemma \ref{lem4.2} and Proposition \ref{prop4.3}, we have sub-cylinders $A_n^u$ of $\tilde f^n(A_0^u)$ $C^{1}$ converging to $\ell_{\tilde f}^u$ as $n\to \infty$. See Figure \ref{fig6_4}.
Similarly, we have sub-cylinders $A_n^s$ of $\tilde f^{-n}(A_0^s)$ 
$C^{1}$ converging to a segment in $W_{\loc}^s(\tilde q_2,\tilde f)$ centered at $\tilde X$ as $n\to \infty$.
Hence, 
by making a $C^{2}$ perturbation of $\tilde f$ similar to the above one  in a small neighborhood of $\tilde X$, 
we can obtain a diffeomorphism $\tilde g$ arbitrarily $C^{2}$ close to $\tilde{f}$ 
such that 
$W^{u}(\tilde{q}_{1}, \tilde g)$ and $W^{s}(\tilde{p}_{1}, \tilde g)$ have a heterodimensional tangency of elliptic type $r$, 
while $L_{\tilde g}^u$ is still contained  in  $|\tilde{\mathcal{D}}|$. 
See Figure \ref{fig6_5}-(a).
\medskip

%%%%%%%%%%%%%%%%%%%%%%%%%%%%%%%%%%%%%%%
\begin{figure}[hbt]
\centering
\scalebox{0.8}{\includegraphics[clip]{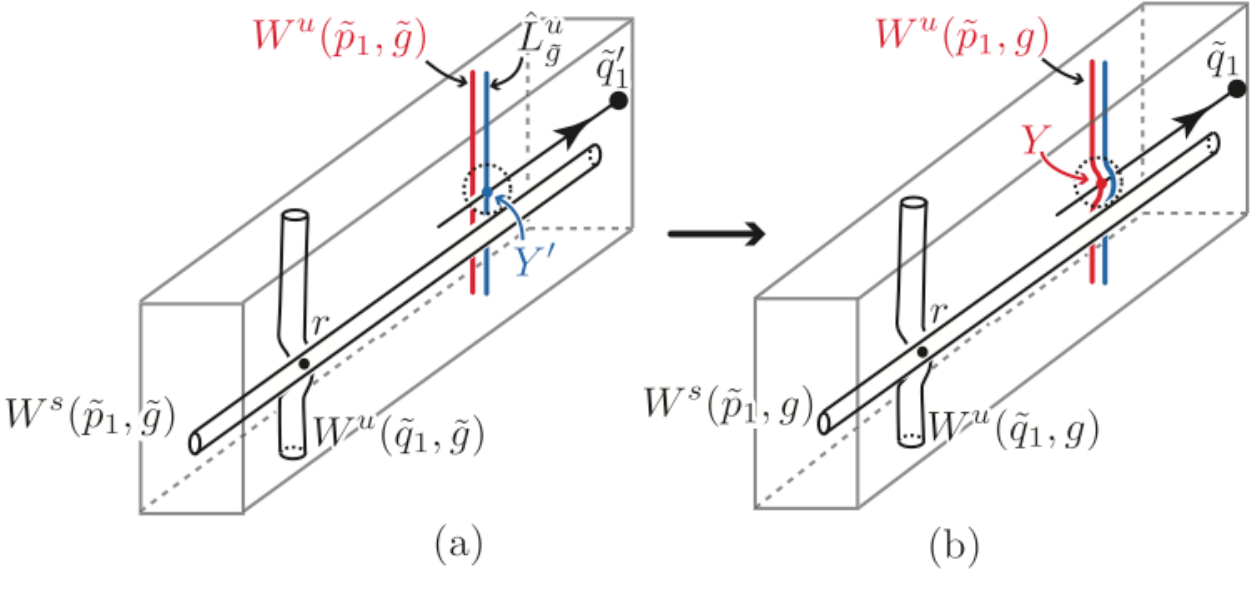}}
\caption{The last perturbation around $Y^{\prime}$.}
\label{fig6_5}
\end{figure}
%%%%%%%%%%%%%%%%%%%%%%%%%%%%%%%%%%%%%%%

\noindent
\textbf{(III) The 3rd perturbation from  $\tilde g$ to $g$.}
Since $\Lambda_{1}$ is transitive and contains $\tilde p_1$, $\tilde p_2$, 
there exists a segment $\hat L_{\tilde g}^u$ in $W^u(\tilde p_1,\tilde g)$ arbitrarily $C^2$ close 
to $L_{\tilde g}^u$ and hence it is also contained  in  $|\tilde{\mathcal{D}}|$.  
That is, 
there exists a constant $\delta_{0}>0$ and 
a $C^{2}$ embedding  $D: [-1,1] \times[-\delta, \delta]^{d-3} \to \Delta$ 
such that 
$$D([-1,1] \times\{\boldsymbol{0}^{d-3}\} )= \hat L_{\tilde g}^u,\quad 
D_{\delta}(\hat L_{\tilde g}^u):=D([-1,1] \times[-\delta, \delta]^{d-3}) \in \mathcal{D}$$
for every $0<\delta<\delta_{0}$.
Therefore, 
$$
W^{s}_{\loc}(\Lambda_{0},\tilde g)\cap D_{\delta}(\hat L_{\tilde g}^u)\neq\emptyset.
$$ 
This implies  that there exists a point $\tilde q_1^{\prime}\in \Lambda_0$ such that 
$$
W^{s}_{\loc}(\tilde q_1^{\prime}, \tilde{g})\cap \hat L_{\tilde g}^u\neq \emptyset.
$$
Note that, although $\tilde q_1^{\prime}$ may not be periodic point of $\tilde g$, 
it  is  an accumulation point of periodic points of $\Lambda_{0}$. 
Observe that $W^{s}_{\loc}(\tilde q_1^{\prime}, \tilde{g})\cap \hat L_{\tilde g}^u$ 
consists of a quasi-transverse intersection $Y^{\prime}$. See Figure \ref{fig6_5}-(a).
Thus, 
once again  by  adding a perturbation similar to the above one in a small neighborhood of $Y^{\prime}$,  
we have a diffeomorphism 
$g$ arbitrarily $C^2$ close to $\tilde g$ such that 
$W^u(\tilde p_1,g)$ and $W^s(\tilde q_1,g)$ have a quasi-transverse intersection $Y$.
By the condition (iv), it is possible to choose the perturbation which does not break the heterodimensional  tangency $r$. See Figure \ref{fig6_5}-(b).
Since $g|_{\Lambda_{1}}=\tilde f|_{\Lambda_{1}}$, $\tilde p_1$ is sectionally dissipative 
with respect to $g$. 
It follows that $g$ is a diffeomorphism satisfying all the conditions in Theorem \ref{theoremA}.
Thus, the claim of Theorem \ref{theoremC} follows directly from Theorem \ref{theoremA}.
\end{proof}

\subsection*{Acknowledgments}
The authors would like to thank the referee for valuable comments and suggestions, 
according to which many parts of this paper are improved and corrected.
The first and second 
authors were partially supported by 
JSPS KAKENHI Grant Numbers 
%Grant-in-Aid for Scientific Research (C) 
22540226 and
22540092, respectively.

%%%%%%%%%%%%%%%%%

\end{document}